\RequirePackage{fix-cm}
\documentclass[smallextended]{svjour3}       
\smartqed  
\usepackage{graphicx}
\usepackage{color}
\usepackage{mathptmx}      
\usepackage{latexsym}
\usepackage[all]{xy}

\usepackage{amsmath}
\usepackage{amssymb}
\usepackage{graphicx}
\usepackage{eucal}
\usepackage{mathrsfs}
\usepackage{pifont}
\usepackage{cjhebrew}

\newcommand{\calM}{{\mathcal M}}
\newcommand{\calN}{{\mathcal N}}

\newcommand{\calS}{{\mathcal S}}

\newcommand{\bbR}{{\mathbb R}}

\newcommand{\brk}[1]{\left(#1\right)}          
\newcommand{\Brk}[1]{\left[#1\right]}          

\newcommand{\Cases}[1]{\begin{cases} #1 \end{cases}}

\newcommand{\beq}{\begin{equation}}
\newcommand{\eeq}{\end{equation}}
\newcommand{\bsplit}{\begin{split}}
\newcommand{\esplit}{\end{split}}
\newcommand{\baligned}{\begin{aligned}}
\newcommand{\ealigned}{\end{aligned}}

\newcounter{sect}

\newenvironment{dingnum}{
\begin{dingautolist}{192}\setlength{\itemsep}{-2pt}}{\end{dingautolist}}

\providecommand{\R}{\bbR}

\newcommand{\Textand}{\qquad\text{ and }\qquad}

\newcommand{\e}{\varepsilon}
\newcommand{\W}{\Omega}

\newcommand{\G}{\Gamma}
\newcommand{\tI}{\widetilde{I}}

\newcommand{\vp}{\varphi}

\newcommand{\weakly}[1]{\stackrel{#1}{\rightharpoonup}}

\newcommand{\strongly}[1]{\stackrel{#1}{\longrightarrow}}
\newcommand{\GC}{\strongly{\Gamma}}
\newcommand{\wt}[1]{\widetilde{#1}}

\newcommand{\N}{{\boldsymbol{\calN}}}

\newcommand{\dist}{\operatorname{dist}}

\newcommand{\Vol}{\operatorname{Vol}}

\newcommand{\NW}{W^\perp}

\newcommand{\metric}{\mathfrak{g}}
\newcommand{\euc}{\mathfrak{e}}
\newcommand{\go}{{\metric}}
\newcommand{\tgo}{\tilde{\go}}

\newcommand{\q}{{\mathfrak{q}}}
\newcommand{\qo}{{\q}}

\newcommand{\Volume}{d{\text{vol}_{\go}}}
\newcommand{\VolumeS}{d{\text{vol}_{\go|_\S}}}

\newcommand{\Volumet}{d{\text{vol}_{\tilde\go}}}

\renewcommand{\S}{{\calS}}
\newcommand{\NS}{{\calN\S}}

\newcommand{\id}{\operatorname{Id}}

\newcommand{\Ppar}{{P^\parallel}}

\newcommand{\Pperp}{{P^\perp}}
\newcommand{\ipar}{{\iota^\parallel}}
\newcommand{\iperp}{{\iota^\perp}}

\newcommand{\qperp}{\qo^\perp}
\newcommand{\opar}{\omega^\parallel}
\newcommand{\operp}{\omega^\perp}
\newcommand{\dpar}{D^\parallel}
\newcommand{\dperp}{D^\perp}

\newcommand{\piPush}{\pi_\star}

\newcommand{\im}{\text{Im}}

\def\Xint#1{\mathchoice
{\XXint\displaystyle\textstyle{#1}}%
{\XXint\textstyle\scriptstyle{#1}}%
{\XXint\scriptstyle\scriptscriptstyle{#1}}%
{\XXint\scriptscriptstyle\scriptscriptstyle{#1}}%
\!\int}
\def\XXint#1#2#3{{\setbox0=\hbox{$#1{#2#3}{\int}$ }
\vcenter{\hbox{$#2#3$ }}\kern-.6\wd0}}

\newcommand{\dashint}{\Xint-}

\DeclareMathOperator*{\limarrow}{\longrightarrow}

\renewenvironment{proof}{{\flushleft \emph{Proof}:}}{\hfill\ding{110}}

\numberwithin{equation}{section}
\numberwithin{theorem}{section}
\numberwithin{definition}{section}
\numberwithin{lemma}{section}
\numberwithin{corollary}{section}
\numberwithin{proposition}{section}

\begin{document}

\title{A Riemannian approach to the membrane limit in non-Euclidean elasticity\thanks{
This research was partially supported by the Israel Science Foundation and by the Israel-US Binational Science Foundation.}
}

\titlerunning{The membrane limit in non-Euclidean elasticity}        

\author{Raz Kupferman         \and
        Cy Maor 
}


\institute{Raz Kupferman \at
              Institute of Mathematics, The Hebrew University, Jerusalem Israel 91904 \\
              \email{raz@math.huji.ac.il}           
           \and
           Cy Maor \at
              Institute of Mathematics, The Hebrew University, Jerusalem Israel 91904 \\
              \email{cy.maor@mail.huji.ac.il} }

\date{Received: date / Accepted: date}

\maketitle


\begin{abstract}

Non-Euclidean, or incompatible elasticity is an elastic theory for pre-stressed materials, which is based on a modeling of the elastic body as a Riemannian manifold. In this paper
we derive a dimensionally-reduced model of the so-called membrane limit of a thin incompatible body.
By generalizing classical dimension reduction techniques to the Riemannian setting, we are able to prove a general theorem that applies to an elastic body of arbitrary dimension, arbitrary slender dimension, and arbitrary metric. 
The limiting model implies the minimization of an integral functional defined over immersions of a limiting submanifold in Euclidean space. The limiting energy only depends on the first derivative of the immersion, and for frame-indifferent models, only on the resulting pullback metric induced on the submanifold, i.e., there are no bending contributions.

\end{abstract}

\keywords{Riemannian manifolds \and Nonlinear elasticity \and Incompatible elasticity \and Membranes \and Gamma-convergence}

\section{Introduction}

In recent years there has been a renewed interest in the elastic properties of bodies that have an intrinsically non-Euclidean geometry. The original interest in such systems stemmed from the study of crystalline defects, in which case the intrinsic geometry exhibits singularities; see  Bilby and co-workers \cite{BBS55,BS56}, Kondo \cite{Kon55}, Wang \cite{Wan67}, and Kr\"oner \cite{Kro81}. The motivation for the recent interest in non-Euclidean bodies arises from the study of growing tissues \cite{ESK08,AESK11,AAESK12}, thermal expansion \cite{OY10}, and other mechanisms of differential expansion of shrinkage \cite{KES07}; in all these examples  the intrinsic geometry can be assumed to be smooth.

Mathematically, we model an elastic body as a three-dimensional Riemannian manifold, $(\calM,\go)$, equipped with an energy function that assigns an energy to every configuration $f:\calM\to\R^3$ of the manifold into the ambient Euclidean space, $(\R^3,\euc)$. This energy is a measure of the strain, i.e., of the deviation of the pullback metric $f^\star\euc$ from the intrinsic metric $\go$. The body is said to be non-Euclidean if the intrinsic metric has non-zero Riemannian curvature, in which case it cannot be immersed isometrically in three-dimensional Euclidean space. The elastostatic problem consists of finding the configuration $f$ that minimizes the elastic energy given possibly boundary conditions and external forces.

A central theme in material sciences is the derivation of dimensionally-reduced models, which are applicable to elastic bodies that display one or more slender axes. In such models the elastic body is viewed as a lower-dimensional limit of thin bodies (which can be viewed as the mid-surface).
The derivation of dimensionally reduced models goes back to 
Euler, D. Bernoulli, Cauchy, and Kirchhoff \cite{Kir50}, and in the last century, to name just a few, to von Karman \cite{Kar10}, E. and F. Cosserat, Love \cite{Lov27}, and Koiter \cite{Koi66}. 

Dimensionally-reduced models are commonly classified according to two main criteria: the dimension of the limiting manifold (which may be either 1 or 2) and the energy scaling of the reduced energy functional. Plates and shells are examples of two-dimensional reduced models in which the limiting manifold can be embedded in $\R^3$ smoothly enough so that the main energy contribution comes from the second fundamental form (bending effects).  
Membranes are examples of two-dimensional reduced models in which the main energy contribution is from metric deviations of 
the two-dimensional pullback metric from the metric of the limiting manifold (stretching effects). Rods are examples of one-dimensional reduced models. 

Until about 20 years ago, dimension reduction analyses were based  essentially on formal asymptotic methods and uncontrolled ansatzes. 
The rigorous derivation of dimensionally-reduced models was first achieved in the Euclidean case, where the bodies have a natural rest configuration with respect to which deviations can be measured. The membrane limit was derived by Le Dret and Raoult \cite{LR95,LR96}, and generalized by Braides et al. \cite{BFF00} and Babadjian and Francfort \cite{BF05}, whereas the plate and shells limits were derived by James et al. \cite{FJM02b} and \cite{FJMM03}. The rod limit was derived by Mora and M\"uller \cite{MM03}. For non-Euclidean bodies the limiting plate theory was derived by Lewicka and Pakzad \cite{LP10}, whereas Kupferman and Solomon \cite{KS14} proved a general theorem that yields plate, shell and rod limits in non-Euclidean cases.  
All the above mentioned work relies on $\G$-convergence techniques \cite{Dal93}.

In this work we derive the membrane limit of non-Euclidean elasticity. A typical application of such limit would be the study of a thin plant tissue under stretching conditions. We consider here pure displacement problems without body forces; the inclusion of external forces and/or surface traction is not expected to involve any complications \cite{LR95}.

We now describe our main results;  precise definitions and formulations are given in the next section. 
We denote by $\W_h$ a family of $n$-dimensional submanifolds of an $n$-dimensional manifold $(\calM,\go)$ that converge to an $(n-k)$-dimensional manifold $\S$; here $h$ is the thickness of the domains. With every configuration $f_h:\W_h\to\R^n$ (which is regular enough and satisfies the boundary conditions) we associate an energy
\[
I_h(f_h) = \dashint_{\W_h} W(df_h)\,\Volume,
\]
where $\dashint$ denotes volume average, and $W$ is an $h$-independent energy density satisfying some regularity, growth and coercivity conditions, as well as a homogeneity condition. Considering  pure displacement problems, we prove that  $I_h$ $\G$-converges as $h\to0$ to a functional that assigns, to regular enough configurations $F:\S\to\R^n$ that satisfy the boundary conditions, an energy
\[
I(F) = \dashint_\S QW_0(dF)\,\VolumeS,
\]
where $QW_0$ is the quasiconvex envelope of a  projection of the restriction of $W$ to $\S$.
Moreover, every sequence $f_h$ of (possibly approximate) minimizers of $I_h$ has a subsequence that converges to a  minimizer of $I$. 

The basic tools  are the analytic techniques developed in \cite{LR95} along with the geometric framework developed in \cite{KS14}. 
The main difference between our analysis and that in \cite{LR95} is that the current analysis applies to an arbitrary Riemannian setting and to arbitrary dimensions. As such, it does not distinguish a priori between ``plate-like",  ``shell-like"  or ``rod-like" bodies, and neither between Euclidean and non-Euclidean geometries. 

Moreover, the Riemannian setting reveals the geometric content of classical notions in elasticity and analysis.
It requires the revision and the generalization of the notions of quasiconvexity, measurable selection theorems, and Carath\'eodory functions. 
In additon, the material science notion of homogeneity needs to be reinterpreted, which leads to new insights into its geometric meaning. 
Finally, the geometric analysis ``toolbox" constructed in \cite{KS14} is expanded to treat different function spaces and more general energy densities.

\section{Problem statement and main results}

\subsection{Modeling of slender bodies}

We start by presenting the general geometric framework.
Let $\calM$ be a smooth $n$-dimensional oriented manifold; let $\S\subset\calM$ be a smooth
 $m$-dimensional compact oriented submanifold with Lipschitz-continuous boundary; let $k$ denote the codimension of $\S$ in $\calM$.
We endow $\calM$ with a metric $\go$, and denote the induced metric on $\S$ by $\go|_\S$.

We view $T\S$ as a sub-bundle of $T\calM|_\S$, and denote its orthogonal complement, the normal bundle of $\S$ in $\calM$, by $\NS$, so that
\[
T\calM|_\S \cong T\S \oplus \NS.
\]
 
Let $h$ be a continuous positive parameter, and define a family of tubular neighborhoods of $\S$ by
\[
\W_h = \{ p\in \calM : \dist(p,\S)<h \}\subset \calM.
\]
These tubular neighborhoods inherit the metric $\go$. 
Our smoothness and compactness assumptions on $\S$ imply that for small enough $h$ (say $h\in(0,h_0]$ for some $h_0>0$) the exponential map,
\[
\exp : \{ (p,\xi) \in \NS : |\xi| < h \} \to \W_h
\]
is a diffeomorphism between an open subset of $\NS$ and $\W_h$.
Therefore, we have a structure of fiber bundle $\pi:\W_h \to \S$, with the fiber being a $k$-dimensional ball, and the projection $\pi$ is obtained by moving along the geodesic from a point $p\in \W_{h}$ to its nearest neighbor in $\S$.

\subsection{Configurations and boundary conditions}

We view $\W_h$ as a family of (shrinking) bodies. 
A \emph{configuration} of $\W_h$ is a map $f_h:\W_h\to\R^n$ from the so-called material manifold $(\W_h,\go)$ to the physical space $(\R^n,\euc)$, where $\euc$ is the Euclidean metric.

In the elastic context, 
we consider a \emph{pure displacement} problem, where the boundary conditions are imposed on the ``outer ring" of $\W_h$, 
\[
\G_h = \{\xi\in\W_h:\,\, \pi(\xi) \in\partial\S\}.
\]
A sketch of the manifolds $\calM$, $\W_h$, $\S$ and the boundary manifolds $\G_h$ and $\partial\S$ are shown in Figure~\ref{fig:0}. 

\begin{figure}
\begin{center}
\includegraphics[height=2.4in]{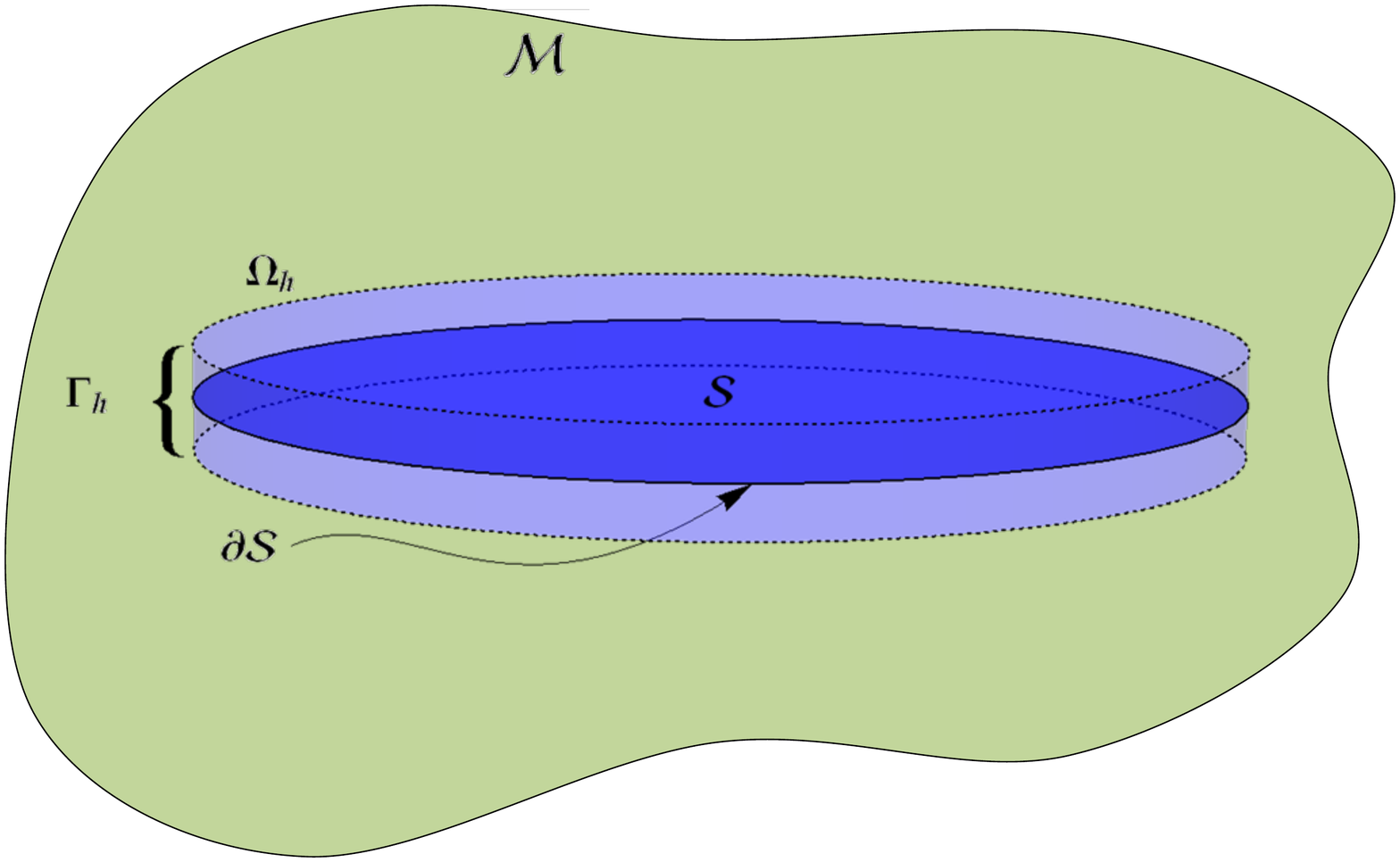}
\end{center}
\caption{Schematic illustration of the manifolds $\calM$, $\W_h$, $\S$, $\G_h$ and $\partial\S$.}
\label{fig:0}
\end{figure}

We impose the boundary conditions by specifying a mapping of $\partial\S$ into $\R^n$, and extending it linearly to $\Gamma_h$ via a mapping of normal vectors. 
Specifically, 
let $F_{bc}$ be a mapping $\partial\S\to \R^n$ and $\qperp_{bc}$ be a section of $(\NS^*\otimes\R^n)|_{\partial\S}$ (an assumption on the regularity of these mappings will be imposed later). A mapping $f_h:\W_h\to\R^n$ satisfies the boundary conditions if
\beq
f_h(\xi) = F_{bc}(\pi (\xi)) + (\qperp_{bc})_{\pi(\xi)}(\xi)   \qquad  \xi\in\G_h,
\label{eq:f_bc}
\eeq
where we identify $\xi\in\G_h$ with its image under the diffeomorphism of $\W_h$ to an open set in $\NS$.

The condition \eqref{eq:f_bc} can be written in a more compact form. For a section $\qperp\in\G(\S;\NS^*\otimes\R^n)$, the pullback $\pi^*\qperp$ is a section in $\G(\W_h;\pi^*\NS^*\otimes\R^n)$, and
\[
(\pi^*\qperp)_\xi(\eta) = \qperp_{\pi(\xi)}(\eta),
\]
where $\xi\in\W_h$ and $\eta\in(\pi^*\NS)_\xi \cong \NS_{\pi(\xi)}$. With these identifications we can write $(\qperp_{bc})_{\pi(\xi)}(\xi) = (\pi^*\qperp_{bc})_\xi(\xi)$.  Introducing the tautological section $\lambda\in\G(\W_h;\pi^*\NS)$ defined by
\[
\lambda_\xi = \xi\in (\pi^*\NS)_\xi,
\]
we can write \eqref{eq:f_bc} as follows:
\beq
f|_{\G_h} = F_{bc}\circ\pi + \pi^*\qperp_{bc}\circ\lambda.
\label{eq:f_bc2}
\eeq

\subsection{The energy functional}

The assumption whereby the bodies $\W_h$ are 
hyper-elasticity means that to each admissible (in a sense to be made precise below) configuration $f_h$ corresponds an elastic energy of the form
\beq
\dashint_{\W_h} W(df_h)\,\Volume,
\label{eq:E}
\eeq
where
\[
W : T^*\W_{h_0}\otimes\R^n \to \R.
\]
is an \emph{elastic energy density}; note that $W$ is independent of $h$.

For $q\in T^*\W_{h_0}\otimes\R^n$, we denote by $|q|$ the norm that is inherited from both $\go$ and $\euc$. 
We assume  $W$ to be continuous, and that there exists a $p\in(1,\infty)$ such that:
\begin{enumerate}
\item Growth condition: $|W(q)| \le C(1+| q |^p)$,
\item Coercivity: $W(q) \ge \alpha| q|^p - \beta$,
\item Lipschitz property: for $q,q'\in T_x^*\W_{h_0}\otimes\R^n$,
\[
|W(q)-W(q')| \le C(1+ | q|^{p-1} + |q'|^{p-1})|q-q'|,
\]
\item Homogeneity over fibers, which will be defined in the next section. 
When $(\W_h,\go)$ is Euclidean, this condition amounts to $W:\W_h\times \R^{n\times n}\to\R$ being in fact a mapping $\S\times\R^{n\times n}\to\R$, i.e., the spatial dependence of the energy density only depends on the projection on the mid-surface.   
\end{enumerate}
Under these conditions, the total elastic energy \eqref{eq:E} is defined for all $f_h\in W^{1,p}(\W_h;\R^n)$.

A prototypical energy density that satisfies these conditions for $p=2$ is 
\beq
W(\cdot)=\dist^2(\cdot, \text{SO}(\W_{h_0};\R^n)),
\label{eq:Wproto}
\eeq
where $\text{SO}(\W_{h_0};\R^n)$ denotes the metric and orientation preserving transformations $T\W_{h_0} \to \R^n$. This energy density measures how far is a local configuration  from being a local isometry.
Note however that we do not assume $W$ to satisfy frame-indifference or isotropy.

An example for a density that is non-homogeneous over fibers can be obtained by multiplying $W$ in \eqref{eq:Wproto} by a non-constant function $f:\W_{h_0} \to (0,\infty)$ that depends only on the distance from $\S$. This energy density also measures how far is a local configuration from being a local isometry. However, the measure is different in different ``layers". The membrane limit in this case is different from the one addressed in this paper; see discussion.

Note also that the more physical case where $\lim_{\det(q)\to 0} W(q) = \infty$ and $W(q) = \infty$ for singular or orientation reversing transformations is not covered by the analysis in this paper, as it does not satisfy the growth condition.

The space of admissible configurations is defined by requiring that \eqref{eq:E} be well-defined, as well as the satisfaction of the boundary conditions \eqref{eq:f_bc2}. We denote:
\[
W_{bc}^{1,p}(\W_h;\R^n) = \{ f\in W^{1,p}(\W_h;\R^n) :  f|_{\G_h} = F_{bc}\circ\pi + \pi^*\qperp_{bc}\circ\lambda\}.
\]
We assume that $F_{bc}$ and $\qperp_{bc}$ are regular enough such that the spaces $W_{bc}^{1,p}(\W_h;\R^n)$ are not empty for small enough $h$. Note that each $W_{bc}^{1,p}(\W_h;\R^n)$ is an affine space with respect to the vector space $\{ f\in W^{1,p}(\W_h;\R^n) :  f|_{\G_h} = 0 \}$.

For technical reasons it is convenient to extend the domain of the energy functional to configurations $L^p(\W_h;\R^n)$ that may not satisfy either regularity or boundary conditions as follows:  
\beq
I_h(f)=
\begin{cases}
\dashint_{\W_h} W(df)\, \Volume & f\in W_{bc}^{1,p}(\W_h;\R^n) \\ 
\infty & \text{otherwise},
\end{cases}
\label{eq:Ih}
\eeq
where $\dashint$ denotes a volume average, namely
\[
\dashint_{\W_h} \alpha\,\Volume = \frac{\int_{\W_h} \alpha\,\Volume}{\int_{\W_h} \Volume}.
\]

\subsection{Main result}

We now define an energy density for configurations of the mid-surface $F:\S\to\R^n$. 
The restriction $W|_\S$ is a map $T^*\W_h|_\S\otimes\R^n \to \R$, which we may identify with a map
$(T^*\S\oplus\NS^*)\otimes\R^n\to\R$. We then define
\[
W_0 :  T^*\S\otimes \R^n \to \R
\]
as follows:
\[
W_0(q) = \min_{r\in (\NS^*\otimes\R^n)_{\pi(q)}} W|_\S(q\oplus r).
\]
Note that the coercivity condition on $W$ implies that the minimum is indeed attained.
Let $QW_0 :  T^*\S\otimes \R^n \to \R$ be the quasiconvex envelope of $W_0$ (for more details on quasiconvex functions in a Riemannian setting see next section). The growth condition imposed on $W$ implies that $W_0$ and $QW_0$ satisfy similar conditions (see Lemma \ref{lm:w0} and Corollary \ref{cy:qw0} below).

We are now ready to state our main result:

\begin{theorem}
\label{tm:main}
The sequence of functionals $(I_h)_{h\le h_0}$ $\Gamma$-converges in the strong $L^p$ topology,  as $h\to 0$, to a limit $I:L^p(\S;\R^n)\to\R$ defined by:
\[
I(F) = \Cases{\dashint_\S QW_0(dF)\,\VolumeS &  F\in W_{bc}^{1,p}(\S;\R^n), \\
\infty & \text{otherwise},}
\]
where $W_{bc}^{1,p}(\S;\R^n) = \{F\in W^{1,p}(\S;\R^n) : F|_{\partial\S} = F_{bc}\}.$
\end{theorem}

Note that each $I_h$ is defined over a different functional space, which requires either rescaling of $I_h$, or a slight modification in the definition of $\G$-convergence; these approaches are equivalent, and we use the second one (see next section).

The coercivity of the energy density $W$ further enables us to prove the following natural corollary of $\G$-convergence, which implies that $I$ can be viewed an $(n-k)$-dimensional approximation to the $n$-dimensional elastic functional $I_h$ for small $h$, in the following sense:

\begin{corollary}
\label{main cor}
Let $f_h\in W_{bc}^{1,p}(\W_h;\R^n)$ be a sequence of (approximate) minimizers of $I_h$, that is, 
\[
I_h(f_h) = \inf_{L^p(\W_h;\R^n)} I_h(\cdot) + r(h),
\]
where $\lim_{h\to0} r(h) = 0$.
Then $(f_h)$ is a relatively compact sequence (with respect to the strong $L^p$ topology), and all its limits points are minimizers of $I$.
Moreover,
\[
\lim_{h\to0} \inf_{L^p(\W_h;\R^n)} I_h(\cdot) = \min_{L^p(\S;\R^n)} I(\cdot).
\]
\end{corollary}

\section{Preliminaries}

\subsection{Geometric setting}

\subsubsection{Decomposition of $T\calM|_\S$}

As stated in the previous section, we view $\W_h$ as a restriction of $\NS$ via the exponential map, where $\NS$ is the normal bundle of $\S$ in $\calM$; we denote by $\pi$ the projection from $\NS$ or $\W_h$ into $\S$.
 
We define the projection operators
\[
\Ppar : T\calM|_\S \to T\S 
\Textand
\Pperp : T\calM|_\S \to \NS,
\]
and the corresponding inclusions
\[
\ipar : T\S \hookrightarrow T\calM|_\S
\Textand
\iperp : \NS \hookrightarrow T\calM|_\S.
\]

\subsubsection{Pullback bundles}

Let $E\to\S$ and $F\to\NS$ (or $\W_h$) be vector bundles. 
The pullback $\pi^*E$ is a vector bundle over $\NS$, such that for $\xi\in\NS$, the fiber $(\pi^*E)_\xi$ is identified with the fiber $E_{\pi(\xi)}$. Let $\Phi:\pi^*E\to F$, i.e.,
\[
\Phi_\xi : (\pi^*E)_\xi \to F_\xi.
\]
Since $(\pi^*E)_\xi$ is canonically identified with $E_{\pi(\xi)}$, we can unambiguously apply $\Phi_\xi$ to elements of $E_{\pi(\xi)}$.

\subsubsection{Connections and parallel transport}

Let $\nabla$ denote the Levi-Civita connection on $T\calM$, and by abuse of notation, also on its restriction to $T\calM|_\S$. The induced connection on $\NS$ is defined by
\[
\nabla^\perp = \Pperp\circ\nabla\circ\iperp.
\]

Let $\xi\in\W_h$, and denote by $\Pi_\xi$ the parallel transport with respect to $\nabla$ from $T_{\pi(\xi)}\calM$ to $T_\xi\calM$ along the geodesic from $\pi(\xi)$ to $\xi$. That is, $\Pi$ is a bundle map
\[
\Pi : \pi^*T\calM|_\S \to T\W_h,
\]
that satisfies
\[
\go_\xi(\Pi_\xi u, \Pi_\xi v) = \go_{\pi(\xi)}(u,v),
\]
for every $\xi\in\W_h$ and $u,v\in T_{\pi(\xi)}\calM$.

\subsubsection{Homogeneity}

With  parallel transport defined, we can now define ``homogeneity over fibers"  of the energy density $W$, 

\begin{definition}
$W$ is \textbf{homogenous over fibers} if for every $q\in T^*_\xi\W_h\otimes \R^n$,
\[
W_{\xi}(q) = W_{\pi(\xi)}(q\circ \Pi_\xi \circ \Pi_{\pi(\xi)}^{-1}),
\]
($\Pi_\xi \circ \Pi_{\pi(\xi)}^{-1}$ is the parallel transport $T_{\pi(\xi)}\W_h \to T_\xi\W_h$ along the geodesic connecting $\pi(\xi)$ to $\xi$).
Equivalently,
\[
W = \pi^*\left. W\right|_\S \circ \Pi^*.
\]
\end{definition}

In the classical (Euclidean) context, homogeneity of the energy density means that its dependence on the infinitesimal deformation does not depend on position. In a Riemannian setting, such a statement is problematic since there is no canonical identification of the tangent spaces at different points. A natural generalization of homogeneity is invariance under parallel transport. Note however that parallel transport is dependent on the trajectory between the end points, and therefore homogeneity requires an invariance that is independent on the trajectory. In a coordinate system, homogeneity means that the spatial dependence of $W$ is only through the entries $\go_{ij}$ of the Riemannian metric. The prototypical energy density \eqref{eq:Wproto} is an example of such density.

Homogeneity over fibers is a weaker property, which can be defined for tubular neighborhoods of a submanifold. It implies invariance under parallel transport along normal geodesics, while allowing inhomogeneity in the ``spatial" directions.  In  the particular case of a Euclidean metric, homogeneity over fibers means that the energy density does not depend explicitly on the normal coordinate. As such, it is not an intrinsic material property, however it is a sufficient condition for our purposes. 

An immediate consequence of homogeneity over fibers is the following:

\begin{lemma}
\label{lm:ww0}
For every $q\in T^*\calM|_\S \otimes \R^n$,
\[
\left. W\right|_\S (q) \ge W_0(q\circ \ipar),
\]
or equivalently,
\[
\left. W\right|_\S \ge W_0\circ \ipar^*.
\]
Moreover, if $W$ is homogenous over fibers, then for every $q\in T^*_\xi\W_h\otimes \R^n$,
\[
W_\xi(q) \ge (W_0)_{\pi(\xi)}(q\circ \Pi_\xi \circ \Pi_{\pi(\xi)}^{-1} \circ \ipar),
\]
or equivalently,
\[
W \ge \pi^*(W_0 \circ \ipar^*) \circ \Pi^*.
\]
\end{lemma}

\begin{proof}
From the definition of $W_0$, for every $q\in T^*\calM|_\S \otimes \R^n$,
\[
\left. W\right|_\S (q) \ge W_0(q\circ \ipar) = (W_0\circ \ipar^*)(q).
\]
The second part of the lemma is immediate from this inequality and the definition of homogeneity over fibers.
\end{proof}

\subsubsection{Approximating $\Pi$ and $\go$}

We now construct another bundle isomorphism
\[
\sigma \oplus \iota : \pi^*T\S \oplus \pi^*\NS \cong \pi^*T\calM|_\S \to T\W_h
\]
(see Figure~\ref{fig:1})
that satisfies $d\pi\circ\iota = 0$, $d\pi\circ\sigma= \id$ and $\Pi-\sigma \oplus \iota = O(h)$; the relation between $\sigma\oplus\iota$ and $\pi$ makes it simpler to analyze than $\Pi$, and our assumptions on $W$ will imply that $W$ is "almost" homogenous over fibers with respect to the parallel-transport-like map $\sigma \oplus \iota$. We will use this bundle isomorphism to define another metric, $\tgo$, on $\W_h$, such that $\sigma \oplus \iota$ is its parallel transport.
The metric $\tgo$ approximates $\go$ in a sense that will be made precise. We will then repeatedly switch between the two metrics, thus exploiting the simpler structure of $\tgo$. 

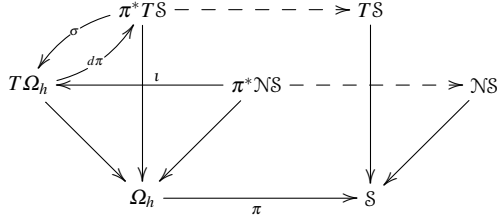
\begin{figure}
\begin{center}
\begin{xy}
(30,5)*+{\W_h} = "Wh";%
(60,5)*+{\S} = "S";%
(45,20)*+{\pi^*\NS} = "piNS";%
(75,20)*+{\NS} = "NS";%
(30,30)*+{\pi^*T\S} = "piTS";%
(60,30)*+{T\S} = "TS";%
(15,20)*+{T\W_h} = "TWh";%
(32,21)*+{_\iota};%
{\ar@{->}_{\pi} "Wh"; "S"};%
{\ar@{->}_{} "piNS"; "Wh"};%
{\ar@{->}_{} "piTS"; "Wh"};%
{\ar@{->}_{} "TWh"; "Wh"};%
{\ar@{->}_{} "NS"; "S"};%
{\ar@{->}_{} "TS"; "S"};%
{\ar@{-->}_{} "piNS"; "NS"};%
{\ar@{-->}_{} "piTS"; "TS"};%
{\ar@{->}_{} "piNS"; "TWh"};%
{\POS"piTS" \ar@/_2ex/ "TWh"|{_\sigma}}
{\POS"TWh" \ar@/_2ex/ "piTS"|{^{d\pi}}}
\end{xy}
\end{center}
\caption{Commutative diagram for $\sigma$, $\iota$ and $d\pi$.}
\label{fig:1}
\end{figure}

Let $\iota : \pi^*\NS \hookrightarrow T\NS$ denote the canonical identification of the vector bundle $\NS$ with its own vertical tangent space. Explicitly, for $\xi\in\NS$ and $\eta\in(\pi^*\NS)_{\xi}$, there is a canonical identification of $\eta$ with an element of $(\NS)_{\pi(\xi)}$. We then define a curve $\gamma:I\to\NS$,
\[
\gamma(t) = \xi + \eta \, t,
\]
and identify $\iota_\xi(\eta) = \dot{\gamma}(0)$.
Clearly $d\pi : T\NS\to \pi^*T\S$ and $\iota$ satisfy: 
\[
d\pi\circ\iota=0.
\]
To fully determine an isomorphism $T\NS \cong \pi^*T\S \oplus \pi^*\NS$ we need a map
\[
\sigma : \pi^*T\S \to T\NS,
\]
such that
\[
d\pi\circ\sigma = \id.
\]
To this end we use the induced connection on $\NS$. Define $\sigma$ to be the unique map such that for any curve $\alpha : I \to S$, and any parallel normal field $\gamma\to\NS$ along $\alpha$, we have
\[
\sigma_{\gamma(0)}(\dot \alpha) = \dot \gamma.
\]
In other words, given $\dot\alpha\in T\S$ and $\xi \in \NS_{\pi(\dot\alpha)}$, $\sigma_\xi(\dot \alpha)$ is the equivalence class of a curve $\gamma : I \to \NS$ along $\alpha$, that satisfies
\[
\gamma(0) = \xi 
\Textand
\nabla^\perp_{\dot\alpha} \dot{\gamma} = 0.
\]
Since $\gamma$ is a curve along $\alpha$, that is, $\pi\circ\gamma = \alpha$, we have, by differentiation, that indeed $d\pi\circ\sigma = \id$.

Note that we defined the range of $\sigma \oplus \iota : \pi^*T\calM|_\S \to T\NS$ to be the total bundle $T\NS$, which means that $\pi$ is viewed as a projection $\pi:\NS\to S$. Restricting $\pi$ to $\W_h$,  we may view $\sigma \oplus \iota$ as a mapping 
$\pi^*T\calM|_\S \to T\W_h$, similar to $\Pi$.

The following lemmas are concerned with the deviation of $\sigma\oplus\iota$ from $\Pi$ and its consequences:

\begin{lemma}
\label{lm:siPi}
The restrictions of $\sigma \oplus \iota$ and $\Pi$ to bundle maps over $\W_h$ (i.e. when $\pi$ is viewed as a mapping $\W_h\to\S$) satisfy
\[
\sigma \oplus \iota - \Pi = O(h).
\]
That is, there exists $C>0$, independent of $h$, such that for every $v\in T\calM|_\S$ and $\xi\in (\W_h)_{\pi(v)}$,
\[
|(\sigma \oplus \iota - \Pi)_\xi(v)| \le Ch |v|
\]
\end{lemma}

\begin{proof}
See Lemma 3.1 in \cite{KS14}.
\end{proof}

An important corollary of Lemma \ref{lm:siPi} is that $W$ is almost homogeneous over fibers with respect to $\sigma\oplus\iota$:
\begin{corollary}
\label{cy:w sigma iota}
There exists $C>0$ such that for every $q\in T^*\W_h\otimes\R^n$,
\[
|W(q) - \pi^*W|_\S \circ (\sigma\oplus\iota)^*(q)| \le Ch(1+|q|^p).
\]
\end{corollary}

\begin{proof}
By the homogeneity over fibers and the Lipschitz property of $W$:
\[
\begin{split}
|W(q) - \pi^*W|_\S \circ (\sigma\oplus\iota)^*(q)| & = | \pi^*W|_\S \circ \Pi^* (q) - \pi^*W|_\S \circ (\sigma\oplus\iota)^*(q)| \\
	& = | \pi^*W|_\S (q\circ \Pi) - \pi^*W|_\S \circ (q\circ (\sigma\oplus\iota))| \\
	& \le C(1 + |q\circ \Pi|^{p-1} + |q\circ (\sigma\oplus\iota)|^{p-1} )| q\circ (\Pi - \sigma\oplus\iota)| \\
	& \le C(1+|q|^{p-1})|q|\cdot h \le Ch(1+ |q|^p).
\end{split}
\]
\end{proof}

Let $\tgo$ denote the unique metric on $\W_h$ such that $\tilde\go|_\S = \go|_\S$ and $\sigma \oplus \iota$ is an isometry. The following corollary follows immediately from Lemma \ref{lm:siPi} (see \cite{KS14} for details).
\begin{corollary}\label{cy:tgg}
\begin{enumerate}
\item
\[
\tilde \go - \go = O(h),
\]
that is,
$|\tgo(u,v) - \go(u,v)| \le Ch |u| |v|$ for every $u,v\in T\W_h$.
\item
\[
\Volumet - \Volume = O(h).
\]
\item
For small enough $h$, the $L^p$ (resp. $W^{1,p}$) norm on $(\W_{h},\go)$ is equivalent to the $L^p$  (resp. $W^{1,p}$) norm on $(\W_{h},\tgo)$. 
\end{enumerate}
\end{corollary}

We now state some further properties of the metric $\tgo$. We show that $\Volumet$ decomposes into a product $\eta\wedge\omega$ where $\eta$ is related to  {the volume form on $\S$ and $\omega$ is a $k$-form}. 
This decomposition will allow us to use repeatedly Fubini's theorem. Moreover, it will be shown to satisfy nice properties upon the rescaling {of} the tubular neighborhoods $\W_h$.
The definitions are given below; for full details and proofs see \cite{KS14}.

Let $E,F \rightarrow M$ be vector bundles and let $\chi: E \rightarrow F$ be a morphism of vector bundles. Denote by $\Lambda^a \chi : \Lambda^a E \rightarrow \Lambda^a F$ the associated vector bundle morphism between the $a^{th}$ exterior powers of $E$ and $F$.
Write
\begin{align*}
\rho  &= (\sigma \oplus \iota)^{-1*} \circ (\pi^*\Ppar)^* : \pi^* T^*\S \rightarrow T^*\W_h \\
\theta  &= (\sigma \oplus \iota)^{-1*} \circ (\pi^*\Pperp)^* : \pi^* \NS^* \rightarrow T^* \W_h.
\end{align*}
Note that
\begin{equation}\label{eq:tris}
\sigma^* \circ \rho = \id, \qquad \iota^* \circ \theta = \id, \qquad \sigma^* \circ \theta = 0, \qquad \iota^* \circ \rho = 0.
\end{equation}
Moreover, equations~\eqref{eq:tris} uniquely characterize $\rho$ and $\theta$.
Taking exterior powers, we have
\[
\Lambda^i \rho : \Lambda^i \pi^* T^*\S \rightarrow \Lambda^i T^*\W_h, \qquad \Lambda^j \theta : \Lambda^j \pi^* \NS^* \rightarrow \Lambda^j T^* \W_h,
\]
{and equations~\eqref{eq:tris} then imply that}
\[
\bigoplus_{i+j = l} \Lambda^i \rho \wedge \Lambda^j\theta  = \Lambda^l (\sigma \oplus \iota)^{-1*}.
\]


Let $\tilde \eta$ be the unit norm section of $\Lambda^{n-k} T^*\S$ belonging to the orientation class, i.e. $\tilde \eta = \VolumeS$. Let $\tilde \omega$ be the unit norm section of $\Lambda^k \NS^*$ belonging to the orientation class determined by the orientations of $\calM$ and $\S$. Define
\[
\eta = \Lambda^{n-k} \rho \circ \pi^*\tilde \eta, \qquad \omega = \Lambda^k \theta \circ \pi^*\tilde \omega.
\]
In particular, $\eta \in A^{n-k}(\W_h)$ and $\omega \in A^k(\W_h)$. It follows from the definition that
\beq\label{eq:vtot}
  \eta \wedge\omega= \Volumet.
\eeq

\begin{lemma}[properties of $\eta$ and $\omega$]
\label{lm:propeo}
Denote by $\pi^\star$ and $\pi_\star$ the pullback and push-forward of forms on $\S$ and on $\W_h$. Then 
\begin{equation*}
\eta  = \pi^\star \VolumeS
\Textand
\pi_\star (\omega) = v_k h^k,
\end{equation*}
where $v_k$ is the volume of the $k$-dimensional unit ball.
\end{lemma}

\begin{lemma}
\label{lm:cm}
We have
\beq
\frac{\piPush \Volume}{|\W_h|} - \frac{\VolumeS}{|\S|} = O(h), \qquad |\S|\nu_k h^k - |\W_h| = O(h^{1+k}).
\label{eq:lim_measure}
\eeq
\end{lemma}

\subsubsection{Rescaling of tubular neighborhoods}

There is a natural notion of strong convergence for functions defined over shrinking tubular neighborhoods;  see next subsection. The notion of  weak convergence turns out to be more subtle. To define it we will need to rescale the tubular neighborhoods (for more details and proofs, see \cite{KS14}).

Define the rescaling operator $\mu_h : \Omega_{h_0} \rightarrow \Omega_{h_0 h}$ by
\[
\mu_h(\xi) = h \xi.
\]
Clearly $\pi \circ \mu_h = \pi$. We assume that $h_0$ is small enough such that part 3 in Corollary \ref{cy:tgg} holds.

\begin{lemma}
\label{lm:resc}
\begin{enumerate}
\item
For every $\xi\in \W_{h_0}$,
\[
d\mu_h\circ \sigma_\xi = \sigma_{h\xi}, \qquad d\mu_h\circ \iota_\xi = h \,\iota_{h\xi}.
\]
\item
\[
\mu_h^\star \omega = h^k \omega, \qquad \mu_h^\star \eta = \eta.
\]
\item
Let $f \in L^1(\Omega_{h_0h})$. Then
\[
\dashint_{\Omega_{h_0h}} f\, \Volume = \frac{1+ O(h)}{\nu_k h_0^k |\S|} \int_{\Omega_{h_0}} (f \circ \mu_h)\, \eta\wedge\omega.
\]
\end{enumerate}
\end{lemma}

\subsection{Convergence in tubular neighborhoods}

We start by defining strong convergence over shrinking domains:

\begin{definition}
\label{df:conv}
Let $f_h\in L^p(\W_{h};\R^n)$ and $F \in L^p(\S;\R^n)$. We say that $f_h\to F$ in the strong $L^p$ topology if
\[
\lim_{h\to0}\,\,\, \dashint_{\W_{h}} |f_h - F\circ\pi|^p\,\Volume = 0.
\]
In other words, defining
\[
\|f_h\|^p_{L^p(\W_{h};\R^n)} = \dashint_{\W_{h}} |f_h|^p\,\Volume,
\]
$f_h\to F$ if for every $\e>0$
\[
f_h \in B_\e(F\circ\pi) 
\]
for every small enough $h$.
\end{definition}

We now discuss the relations between both strong and weak convergence in tubular neighborhoods and ``standard" convergence of mappings rescaled to mappings over $\W_{h_0}$. The following lemma establishes the equivalence between both notions for strong convergence:

\begin{lemma}
\label{lm:str lp conv}
Let $f_h\in L^p(\W_{h_0h};\R^n)$ and $F \in L^p(\S;\R^n)$. Then,
$f_h\to F$ in $L^p$ if and only if $f_h\circ\mu_h \to F\circ\pi$ in the strong $L^p(\W_{h_0};\R^n)$ topology.
\end{lemma}

\begin{proof}
It follows from Lemma \ref{lm:resc} and the relation $\pi \circ \mu_h = \pi$ that
\[
\begin{split}
\dashint_{\W_{h_0h}} |f_h - F\circ\pi|^p\,\Volume &= 
\frac{1+ O(h)}{\nu_k h_0^k |\S|} \int_{\Omega_{h_0}} (|f_h - F\circ\pi|^p \circ \mu_h)\, \omega \wedge \eta =\\
	&= \frac{1+ O(h)}{\nu_k h_0^k |\S|} \int_{\Omega_{h_0}} |f_h\circ \mu_h - F\circ\pi|^p\, \omega \wedge \eta.
\end{split}
\]
Hence $f_h \to F$ if and only if the function $f_h\circ\mu_h \to F\circ\pi$ in $L^p(\W_{h_0};\R^n)$ with respect to the metric $\tgo$. 
Since $L^p$ convergence in $(\W_{h_0},\go)$ is equivalent to $L^p$ convergence in $(\W_{h_0},\tgo)$ (Corollary \ref{cy:tgg}), the proof is complete.
\end{proof}


To address weak convergence we first need a technical lemma, {which basically states that if the normal derivative of a mapping is zero, than the mapping does not depend on the normal coordinate.}

\begin{lemma}
\label{lm:dfi=0}
Suppose that $f\in W^{1,p}(\W_h;\R^n)$ satisfies $df\circ \iota=0$. 
Then, there exists an $F\in W^{1,p}(\S;\R^n)$ such that $f= F\circ \pi$. 
\end{lemma}

\begin{proof}
Let $p\in\S$ be given and let
$\gamma:I\to \pi^{-1}(p)$ be a curve in the fiber of $\W_h$ above $p$.
Since $\pi\circ\gamma = p$, it follows that
\[
d\pi(\dot{\gamma}) = 0.
\]
Since $d\pi \circ \iota = 0$ and $\dim \im \iota = \dim \ker (d\pi) = k$, we have that $\im \iota = \ker (d\pi)$, therefore $\dot{\gamma} \in \im \iota$, and since $df\circ\iota$ is follows that
\[
df(\dot{\gamma}) = 0,
\]
hence $f\circ\gamma = \text{const}$, and there exists an $F:\S\to\R^n$ such that $f = F\circ\pi$.

Note that $F$ can be expressed as
\[
F = \frac{\pi_\star(F\circ\pi\,\omega)}{\pi_\star(\omega)} = \frac{\pi_\star(f\,\omega)}{\pi_\star(\omega)} = \frac{1}{\nu_k h^k} \pi_\star(f\,\omega),
\]
hence $F\in W^{1,p}(\S;\R^n)$.
\end{proof}

{The following lemma generalizes to sequences over tubular neighborhoods the classical fact that bounded sequences in a reflexive Banach space have a weakly compact subsequence.}

\begin{lemma}
\label{lm:wk w1p conv}
Suppose that $f_h\in W^{1,p}(\W_{h_0h};\R^n)$ is a  uniformly bounded sequence (each $f_h$ with its respective volume-averaged norm).
Then there exists a $F \in W^{1,p}(\S;\R^n)$ and a subsequence $f_{h_n}$ such that $f_{h_n}\circ \mu_{h_n} \weakly{} F\circ\pi$ in $W^{1,p}(\W_{h_0};\R^n)$. In particular, $f_{h_n}\to F$ in the strong $L^p$ topology.
\end{lemma}

\begin{proof}
Part 3 of Lemma \ref{lm:resc} implies that a sequence $y_h \in L^p(\W_{h_0h})$ is uniformly bounded if and only if the sequence $y_h\circ \mu_h$ is uniformly bounded in $L^p(\W_{h_0})$. Therefore the boundedness of $f_h$ in $W^{1,p}(\W_{h_0h};\R^n)$ implies that  $f_h\circ\mu_h$ is uniformly bounded in $L^p(\W_{h_0};\R^n)$ and that $\int_{\W_{h_0}} |df_h|^p\circ\mu_h\, \eta\wedge\omega$ is uniformly bounded.


For $f\in  W^{1,p}(\W_{h_0h};\R^n)$ and $\xi\in\W_{h_0}$,
\[
\begin{split}
|df|^2_{\tgo} \circ\mu_h (\xi) &= \tgo_{h\xi}(d_{h\xi}f,d_{h\xi}f) 
= \go_{\pi(\xi)}(d_{h\xi}f\circ(\sigma\oplus\iota)_{h\xi},d_{h\xi}f\circ(\sigma\oplus\iota)_{h\xi})=\\
	&= \go_{\pi(\xi)}(d_{h\xi}f\circ(\sigma\oplus0)_{h\xi},d_{h\xi}f\circ(\sigma\oplus0)_{h\xi})+\\
	&\quad + \go_{\pi(\xi)}(d_{h\xi}f\circ(0\oplus\iota)_{h\xi},d_{h\xi}f\circ(0\oplus\iota)_{h\xi}).
\end{split}
\]
The last equation follows from the definition of the inner product on the cotangent bundle and the fact that $(d_{h\xi}f\circ(0\oplus\iota)_{h\xi})^\sharp\in \NS$ and that $(d_{h\xi}f\circ(\sigma\oplus0)_{h\xi})^\sharp\in T\S$.
Since $\int_{\W_{h_0}} |df_h|^p\circ\mu_h\, \eta\wedge\omega$ is uniformly bounded, it follows that
\[
\int_{\W_{h_0}}(\go_{\pi(\xi)}(d_{h\xi}f\circ(\sigma\oplus0)_{h\xi},d_{h\xi}f\circ(\sigma\oplus0)_{h\xi}))^{p/2}\, \eta\wedge\omega(\xi),
\]
and
\[
\int_{\W_{h_0}}(\go_{\pi(\xi)}(d_{h\xi}f\circ(0\oplus\iota)_{h\xi},d_{h\xi}f\circ(0\oplus\iota)_{h\xi}))^{p/2}\, \eta\wedge\omega(\xi)
\] 
are also uniformly bounded.
On the other hand, part 1 of Lemma \ref{lm:resc} implies that
\begin{equation}
\label{eq:dfmu}
\begin{split}
|d(f\circ\mu_h)|^2_{\tgo} (\xi) &= \tgo_{\xi}(d_{h\xi}f\circ d_\xi\mu,d_{h\xi}f\circ d_\xi\mu)=\\ 
	&= \go_{\pi(\xi)}(d_{h\xi}f \circ d_\xi\mu \circ(\sigma\oplus\iota)_{\xi},d_{h\xi}f \circ d_\xi\mu \circ(\sigma\oplus\iota)_{\xi})=\\
	&= \go_{\pi(\xi)}(d_{h\xi}f \circ(\sigma\oplus h\iota)_{h\xi},d_{h\xi}f \circ(\sigma\oplus h\iota)_{h\xi})=\\
	&= \go_{\pi(\xi)}(d_{h\xi}f\circ(\sigma\oplus0)_{h\xi},d_{h\xi}f\circ(\sigma\oplus0)_{h\xi})+\\
	&\quad+ h^2\go_{\pi(\xi)}(d_{h\xi}f\circ(0\oplus\iota)_{h\xi},d_{h\xi}f\circ(0\oplus\iota)_{h\xi}),
\end{split}
\end{equation}
and therefore $\int_{\W_{h_0}} |d(f_h\circ \mu_h)|^p_{\tgo}\, \eta\wedge\omega$ is uniformly bounded, hence $f_h\circ \mu_h$ is a bounded sequence in $W^{1,p}(\W_{h_0};\R^n)$.
It follows that $f_h\circ \mu_h$  has a subsequence weakly convergent  in $W^{1,p}(\W_{h_0};\R^n)$ (recall that by Corollary \ref{cy:tgg} convergence with respect to $\go$ is equivalent to convergence with respect to $\tgo$); denote the limit by $f$.
Equation \eqref{eq:dfmu} implies that $\int_{\W_{h_0}} |d(f_h\circ\mu_h)\circ \iota|^p \eta\wedge\omega = O(h^p)$, hence $df\circ \iota = 0$ a.e.

Applying Lemma \ref{lm:dfi=0}, there exists {an} $F \in W^{1,p}(\S;\R^n)$ such that $f=F\circ\pi$.
Therefore, $f_{h_n}\circ \mu_{h_n} \weakly{} F\circ\pi$ in $W^{1,p}(\W_{h_0};\R^n)$, and in particular $f_{h_n}\circ \mu_{h_n} \to F\circ\pi$ in $L^{p}(\W_{h_0};\R^n)$. By Lemma \ref{lm:str lp conv}, $f_{h_n}\to F$ in the strong $L^p$ topology.
\end{proof}

\subsection{$\Gamma$-convergence}

The main theorem of this paper is concerned with $\Gamma$-convergence of functionals $I_h: L^p(\W_h;\R^n) \to \bar{\R}$ to a functional $I: L^p(\S;\R^n)  \to \bar{\R}$. Since the standard definition of $\Gamma$-convergence requires the functionals to be defined on the same space, we need a definition of $\Gamma$-convergence over shrinking tubular neighborhoods, and establish its properties.
The proof of properties satisfied by $\G$-convergence over shrinking domains is essentially the same as for fixed domains, and will therefore be omitted; see \cite{Dal93} for details.

\begin{definition}
Let $I_h: L^p(\W_h;\R^n) \to \bar{\R}$ and $I: L^p(\S;\R^n)  \to \bar{\R}$. We will say that $I_h$ $\G$-converges to $\tI$ in the strong $L^p$ topology if
\begin{dingnum}
\item Lower-semicontinuity: for every sequence $f_h\to F$,
\[
I(F) \le \liminf_{h\to0} I_h(f_h).
\]
\item Recovery sequence: for every $F$ there exists a  sequence $f_h\to F$ such that
\[
I(F) = \lim_{h\to0} I_h(f_h).
\]
\end{dingnum}
\end{definition}

An equivalent topological definition is given by the following lemma:

\begin{lemma}
\label{lem:equivalence}
$I_h \GC I$ if and only if for every $F$,
\[
\lim_{\e\to0} \liminf_{h\to0} \inf_{B_\e(F\circ\pi)} I_h(\cdot) = 
\lim_{\e\to0} \limsup_{h\to0} \inf_{ B_\e(F\circ\pi)} I_h(\cdot) = I(F).
\]
\end{lemma}

\begin{proposition}[Urysohn's property]
\label{pr:ury-gc}
If for every sequence $h_n\to0$ there exists a subsequence $h_{n_k}$ such that
\[
I_{h_{n_k}} \GC I,
\]
then $I_{h} \GC I$.
\end{proposition}

\begin{proposition}[Sequential compactness]
\label{pr:comp-gc}
Every sequence of functionals $I_h:L^p(\W_h;\R^n)\to\bar{\R}$ has a $\G$-converging partial limit
$I: L^p(S;\R^n)\to\bar{\R}$.
\end{proposition}

Propositions \ref{pr:ury-gc} and \ref{pr:comp-gc} immediately imply the following corollary:
\begin{corollary}
\label{cy:gc}
$I_h\GC I$ if and only if $I$ is the limit of every $\Gamma$-convergent subsequence.
\end{corollary}

\subsection{Quasiconvexity}

\begin{definition}
\label{df:QC}
Let $(\calM,\go)$ be a Riemannian manifold, and let $U: T^*\calM\otimes\R^m \to \R$ be a fiber-wise locally integrable function. 
We say that $U$ is \textbf{quasiconvex} if for every $p\in \calM$,  every $A\in T_p^*\calM\otimes\R^m$,   every open bounded set $D_p\subset T_p\calM$, and every $\vp\in C_0^\infty (D_p; \R^m)$,
\beq
U(A) \le \dashint_{D_p} U(A + d\phi\circ\kappa)\,\omega_p,
\label{eq:QCdef}
\eeq
where $\kappa:\pi^*T_p\calM\to TT_p\calM$ is the canonical identification, and $\omega_p$ is the $n$-form
\[
\omega_p = \left. \Volume\right|_{p}\circ \Lambda^n \kappa^{-1}.
\]
\end{definition}
The integrand in \eqref{eq:QCdef} reads as follows: for $\xi\in D_p\subset T_p\calM$ and $\eta\in T_p\calM$,
\[
[d\phi\circ\kappa(\xi)](\eta) =  d_\xi\phi\circ\kappa_\xi(\eta)= d_\xi\phi([\xi + \eta t]) \in \R^m,
\]
hence $U(A + d\phi\circ\kappa)$ is indeed a map from $D_p$ to $\R^m$.
As for the volume form $\omega_p$, choosing a coordinate system $(x_1,\ldots,x_n)$ in $\calM$, and denoting by $(\partial_1,\ldots,\partial_n)$ the corresponding coordinates on $T_p\calM$, we have that $\omega_p$ is of the form
\[
\omega_p = \sqrt{\det \go_{ij}(p)}\, d_p x_1\circ \kappa^{-1} \wedge \dots \wedge d_p x_n\circ \kappa^{-1} = \sqrt{\det \go_{ij}(p)}\, d\partial_1 \wedge \dots \wedge d\partial_n.
\]
The constant $\sqrt{\det \go_{ij}(p)}$ cancels {upon} volume average, and therefore the above definition  reduces, after choosing coordinates, to the classical definition of quasiconvexity, see e.g. \cite{AF84} or \cite{Dac89} (it also shows that the definition is independent of the metric).

In the proof of the main theorem we will need the following two relations between quasiconvexity and lower-semicontinuity, 
which are extensions of classical results to the Riemannian setting (see \cite{AF84}); the proofs are in Appendix~\ref{appendix:qc}.
In these theorems, we assume that $(\calM,\go)$ is a Riemannian manifold of finite volume that can be covered by a finite number of charts (note that $\S$ and therefore $\W_h$, satisfy this condition), and that $U: T^*\calM\otimes\R^m \to \R$ is a Carath\'eodory function (see Appendix~\ref{ap:mst}) that satisfies the growth condition
\beq
\label{eq:growthqc}
-\beta \le U(q) \le C(1+|q|^p)
\eeq
for some $\beta,C>0$ and $p\in(1,\infty)$. 
Both theorems  also hold also if $W^{1,p}$ is replaced with $W^{1,p}_0$.

\begin{theorem}
\label{tm:qc1}
Under the above conditions, the functional $I_A:W^{1,p}(A;\R^m)\to \R$ defined by
\[
I_A(f) := \int_A U(df)\,\,\Volume,
\]
where $A\subset\calM$ is an open subset, is weakly sequential lower-semicontinuous for every $A$ if and only if $U$ is quasiconvex.
\end{theorem}

\begin{theorem}
\label{tm:qc2}
Under the above conditions, the weakly sequential lower-semicontinuous envelope of the functional $I_\calM:W^{1,p}(\calM;\R^m)\to \R$ (as defined in the previous theorem) is $\G I_\calM:W^{1,p}(\calM;\R^m)\to \R$ given by
\[
\G I_\calM(f) := \int_\calM QU(df)\,\,\Volume,
\]
where $QU(q)=\sup\{V(q): V\le U \,\,\text{is quasiconvex}\}$ is the quasiconvex envelope of $U$; moreover, $QU$  is a Carath\'eodory quasiconvex function.
\end{theorem}

We now show that Theorem~\ref{tm:qc2} applies to $W_0$:

\begin{lemma}
\label{lm:w0}
$W_0$ is continuous (and in particular Carath\'eodory) and satisfies the same growth and coercivity conditions as $W$ (possibly with different constants).
\end{lemma}

\begin{proof}
The proof of the growth and coercivity condition is the same as in Proposition 1 in \cite{LR95}.

To prove the continuity, we prove that $W_0$ is both lower- and upper-semicontinuous.
Let $q,q_i\in T^*\S\otimes\R^n$ such that $q_i\to q$, and let $r_i \in \NS_{\pi(q_i)}^*\otimes\R^n$ such that 
$W_0(q_i)=W|_\S(q_i\oplus r_i)$.
Let $q_i$ be a subsequence (not relabeled) such that $W_0(q_i)$ converges. The growth property of $W_0$ and the coercivity property of $W$ imply that
\[
\alpha |r_i|^p-\beta \le W|_\S(q_i\oplus r_i) = W_0(q_i) \le C(1+|q_i|^p),
\]
and since $|q_i|$ is a bounded sequence, so is $|r_i|$.
Since $r_i\in (\NS^*\otimes\R^n)_{\pi(q_i)}$ and $\pi(q_i)\to \pi(q)$, there is a convergent subsequence $r_{i_k}\to r\in (\NS^*\otimes\R^n)_{\pi(q)}$.
Therefore,
\[
\lim_{i\to\infty} W_0(q_i) = \lim_{k\to \infty} W|_\S(q_{i_k}\oplus r_{i_k}) \to W|_\S(q\oplus r) \ge W_0(q).
\]
Since this holds for every convergent subsequence of the original sequence $W_0(q_i)$, $W_0$ is lower-semicontinuous.

To prove upper-semicontinuity, let $q,q_i\in T^*\S\otimes\R^n$ such that $q_i\to q$,  let $r \in \NS_{\pi(q)}^*\otimes\R^n$ such that $W_0(q)=W|_\S(q\oplus r)$, and let $\rho$ be a section of $\NS^*\otimes\R^n$ such that $\rho(\pi(q))=r$. Then, by the continuity of $W$,
\[
W_0(q) = W|_\S(q\oplus r) = \lim_{i\to\infty} W|_\S(q_i\oplus \rho(\pi(q_i))) \ge \limsup_{i\to\infty} W_0(q_i),
\]
hence $W_0$ is upper-semicontinuous.
\end{proof}

It follows that Theorems~\ref{tm:qc1} applies to $QW_0$:

\begin{corollary}
\label{cy:qw0}
$QW_0$ is a Carath\'eodory function and satisfies \eqref{eq:growthqc}.
\end{corollary}

\begin{proof}
From Lemma~\ref{lm:w0}, $QW_0(q)\le W_0(q) \le C(1+|q|^p)$ for some $C>0$. Also, observe that the constant function $-\beta$ is a quasiconvex function not larger than $W_0$, hence $QW_0\ge -\beta$.
 Lemma~\ref{lm:w0} also implies that Theorem~\ref{tm:qc2}  applies to $W_0$, and therefore $QW_0$ is a Carath\'eodory function.
\end{proof}

\section{Proof of the main results}

We restate our main Theorem:

\begin{quote}
The sequence of functionals $(I_h)_{h\le h_0}$ define by \eqref{eq:Ih} $\Gamma$-converges to  $I:L^p(\S;\R^n)\to\bar{\R}$ defined by:
\[
I(F) = \Cases{\dashint_\S QW_0(dF)\,\VolumeS &  F\in W_{bc}^{1,p}(\S;\R^n), \\
\infty & \text{otherwise}.}
\]
\end{quote}

To prove this theorem, we
use Corollary \ref{cy:gc}. That is, we prove that $I$ is the limit of every $\G$-converging subsequence of $I_h$. Explicitly,  we assume that $I_h$ is a (not relabeled) $\G$-convergent subsequence with limit $I_0: L^p(\S;\R^n)\to \bar{\R}$, and show that $I_0 = I$. 
The proof is long hence we divide it into several steps: (1) if $F\notin W_{bc}^{1,p}(\S;\R^n)$ then $I_0(F) = I(F)$; (2)  if $F\in W_{bc}^{1,p}(\S;\R^n)$ then $I_0(F) \ge I(F)$; and (3)  if $F\in W_{bc}^{1,p}(\S;\R^n)$ then $I_0(F) \le I(F)$. 
With a slight abuse of notation, we write $\W_h$ instead of $\W_{h_0h}$ whenever rescaling arguments are concerned, as it does not cause confusion and makes the proof more readable.

\paragraph{Step 1: $I_0(F) = I(F)$ when $F$ does not satisfy either the regularity or the boundary conditions}

\begin{proposition}
\label{pr:step1}
If $F\in L^p(\S;\R^n)\setminus W^{1,p}_{bc}(\S;\R^n)$ then $I_0(F)=\infty = I(F)$.
\end{proposition}

\begin{proof}
Let $F\in L^p(\S;\R^n)$ be such that $I_0(F)<\infty$; we will show that $F\in W^{1,p}_{bc}(\S;\R^n)$. 
Let $f_h\to F$ be a recovery sequence, namely,
\[
I_h(f_h) \to I_0(F)<\infty.
\]
We can assume that $I_h(f_h)$ is uniformly bounded by some constant $C<\infty$ for sufficiently small  $h$, hence $f_h\in W^{1,p}_{bc}(\W_h;\R^n)$.
Since $f_h\to F$ in $L^p$, it follows that $\|f_h\|_{L^p}$ is bounded uniformly in $h$.
From the coercivity of $W$, we have
\[
C \ge I_h(f_h) = \dashint_{\W_h} W(df_h)\,\Volume \ge \alpha \dashint_{\W_h} |df_h|^p \,\Volume - \beta,
\]
hence $\|df_h\|_{L^p}$ is also uniformly bounded, hence $f_h$ is uniformly bounded in $W^{1,p}_{bc}(\W_h;\R^n)$.
By Lemma \ref{lm:wk w1p conv}, there is a (not relabeled) subsequence such that $f_h\circ \mu_h \weakly{} F\circ\pi$ in $W^{1,p}(\W_{h_0};\R^n)$. In particular, $F\circ\pi \in W^{1,p}(\W_{h_0};\R^n)$, hence $F\in W^{1,p}(\S;\R^n)$.

It remains to show that $F|_{\partial\S} = F_{bc}$. 
Indeed, since $f_h\in W^{1,p}_{bc}(\W_h;\R^n)$, it is immediate that $f_h\circ \mu_h|_{\G_{h_0}} = F_{bc}\circ\pi + h \pi^*\qperp_{bc}\circ\lambda$. Therefore $f_h\circ \mu_h|_{\G_{h_0}} \to F_{bc}\circ\pi$ uniformly in $\G_{h_0}$. By the continuity of the trace operator as a mapping $W^{1,p}(\W_{h_0};\R^n) \to L^p(\G_{h_0};\R^n)$, we have that $F\circ\pi|_{\G_{h_0}} = F_{bc}\circ\pi$, hence $F|_{\partial\S} = F_{bc}$.
\end{proof}

\paragraph{Step 2: $I_0(F) \ge I(F)$ when $F$ satisfies both the regularity and the boundary conditions}

\begin{proposition}
\label{pn:I_0>I}
$I_0(F)\ge I(F)$ for every $F\in W^{1,p}_{bc}(\S;\R^n)$.
\end{proposition}

\begin{proof}
Let $F\in W^{1,p}_{bc}(\S;\R^n)$ be given. Let $f_h\to F$ be a recovery sequence, namely,
\[
I_0(F) = \lim_{h\to0} I_h(f_h).
\]
If $I_0(F) = \infty$ then the claim is trivial. Otherwise, $I_h(f_h)$ is bounded for sufficiently small $h$, and therefore
$f_h \in W^{1,p}_{bc}(\W_h;\R^n)$ and 
\[
I_h(f_h) = \dashint_{\W_h} W(df_h)\,\Volume.
\]

The coercivity of $W$ implies that $df_h$ is uniformly bounded in $L^p$, hence by {the} Poincar\'e inequality, $f_h$ is uniformly bounded in $W^{1,p}(\W_h;\R^n)$ (note that we need here a version of the Poincar\'e inequality in which the function is prescribed on a subset of the boundary that has positive $(n-1)$-dimensional measure; see Theorem 6.1-8 in \cite{Cia88} for the Euclidean case; the non-Euclidean case is analogous). 

By Lemma~\ref{lm:wk w1p conv}, Lemma~\ref{lm:str lp conv} and the uniqueness of the limit, 
\[
f_h\circ\mu_h\weakly{}F\circ \pi \quad \text{in \,\,$W^{1,p}(\W_{h_0};\R^n)$}.
\]
By Corollary \ref{cy:w sigma iota}, Lemma \ref{lm:ww0}, and the definition of $QW_0$:
\[
\begin{split}
I_h(f_h)  &=  \dashint_{\W_h} W(df_h)\,\Volume \\
&=\dashint_{\W_{h}} \pi^* \left.W\right|_\S \circ (\sigma\oplus\iota)^* (df_h) \,\Volume + O(h) \\
&\ge \dashint_{\W_{h}}  \pi^* W_0 \circ  \pi^*\ipar^* \circ (\sigma\oplus\iota)^*  (df_h) \,\Volume + O(h) \\
&\ge \dashint_{\W_{h}}  \pi^* QW_0 \circ  \pi^*\ipar^* \circ (\sigma\oplus\iota)^*  (df_h) \,\Volume + O(h) \\
&= \dashint_{\W_{h}}  \pi^* QW_0 \circ  \sigma^*  (df_h) \,\Volume + O(h).
\end{split}
\]

The growth condition of $QW_0$ implies that $\pi^* QW_0 \circ  \sigma^*  (df_h) \in L^1(\W_h)$, hence from  the third part of Lemma \ref{lm:resc} we have that 
\beq
I_h(f_h) \ge
\frac{1+O(h)}{\nu_k h_0^k |\S|}\int_{\W_{h_0}} ( \pi^* QW_0 \circ  \sigma^*  (df_h)\circ\mu_h)\, \eta\wedge\omega + O(h).
\label{eq:aaa}
\eeq

We next show that
\beq
\pi^* QW_0 \circ  \sigma^* (df_h)\circ\mu_h = \pi^* QW_0 \circ  \sigma^* (d(f_h\circ\mu_h)).
\label{eq:bbb}
\eeq
Indeed, on the one hand,  for $\xi\in\W$:
\[
\begin{split}
\pi^* QW_0 \circ  \sigma^* (df_h)\circ\mu_h (\xi) &= \pi^* QW_0 (df_h\circ \sigma)(h\xi) \\
&= (QW_0)_{\pi(h\xi)} (d_{h\xi}f_h\circ \sigma_{h\xi}),
\end{split}
\]
whereas from the first part of Lemma \ref{lm:resc}.
\[
\begin{split}
\pi^* QW_0 \circ  \sigma^* (d(f_h\circ\mu_h)) (\xi) & = (QW_0)_{\pi(\xi)} (d_\xi(f_h\circ\mu_h)\circ \sigma_\xi) \\
&=  (QW_0)_{\pi(\xi)} (d_{h\xi} f_h\circ d_\xi \mu_h\circ \sigma_\xi) \\
&=  (QW_0)_{\pi(\xi)} (d_{h\xi} f_h\circ \sigma_{h\xi}).
\end{split}
\]

Substituting \eqref{eq:bbb} into \eqref{eq:aaa}:
\[
I_h(f_h) \ge
\frac{1+O(h)}{\nu_k h_0^k |\S|}\int_{\W_{h_0}}  \pi^* QW_0 \circ  \sigma^* (d(f_h\circ\mu_h))\, \eta\wedge\omega + O(h).
\]

Since $I_0(F) = \lim_{h\to0} I_h(f_h)$, it follows that
\begin{equation}
\label{eq:liminf QW_0}
I_0(F) \ge \liminf_{h\to0}J(f_h\circ \mu_h),
\eeq
where $J: W^{1,p}(\W_{h_0};\R^n) \to \R$ is defined by 
\[
J(f) :=  \frac{1}{\nu_k h_0^k |\S|}\int_{\W_{h_0}}  \pi^* QW_0 \circ  \sigma^* (df)\, \eta\wedge\omega.
\]

Lemma \ref{lm:qw0sigma qc} below shows that the quasiconvexity of $QW_0$ implies the quasiconvexity of $\pi^* QW_0 \circ \sigma^*: T^*\W_{h_0}\otimes\R^n \to \R$. 
By Corollary~\ref{cy:qw0}, $QW_0$ is a Carath\'eodory that satisfies \eqref{eq:growthqc}, and therefore the same holds for $\pi^* QW_0 \circ \sigma^*$, hence, by Theorem~\ref{tm:qc1}, $J$ is sequentially weak lower-semicontinuous in $W^{1,p}(\W_{h_0};\R^n)$.
Since $f_h\circ\mu_h\weakly{}F\circ \pi$ in $W^{1,p}(\W_{h_0};\R^n)$, it follows that 
\[
I_0(F) \ge J(F\circ\pi).
\]
It remain to calculate $J(F\circ\pi)$, 
\[
\begin{split}
J(F\circ\pi) &= \frac{1}{\nu_k h_0^k |\S|} \int_{\W_{h_0}} \pi^* QW_0 \circ  \sigma^* (d(F\circ \pi)) \, \eta\wedge\omega =\\
& = \frac{1}{\nu_k h_0^k |\S|} \int_{\W_{h_0}} \pi^* QW_0(dF) \, \eta\wedge\omega = \\
& = \frac{1}{\nu_k h_0^k |\S|} \int_{\W_{h_0}} \pi^\star (QW_0(dF) \VolumeS)\wedge\omega = \\
& = \frac{1}{\nu_k h_0^k |\S|} \int_\S \pi_\star (\pi^\star (QW_0(dF) \VolumeS)\wedge\omega) = \\
& = \dashint QW_0(dF) \VolumeS = I(F).
\end{split}
\]
where in the passage from the first to the second line we used the fact that  $d\pi \circ \sigma = Id_{\pi^*T\S}$; in the passage from the second to the third line we used the identity $(\pi^* QW_0)(dF)=\pi^*(QW_0(dF))$, Lemma \ref{lm:propeo}  and Fubini's theorem; in the passage form the third to the fourth line we used Fubini's theorem along with the definition of the push-forward operator $\pi_\star$; in the passage from the fourth to the fifth line we used the property of {the} push-forward operator,
\[
\pi_\star(\pi^\star\alpha\wedge\beta) = \alpha\wedge\pi_\star(\beta),
\]
and Lemma~\ref{lm:propeo}. This concludes the proof.
\end{proof}

The following lemma, which is used in the proof {of proposition \ref{pn:I_0>I}}, is a generalization of a property proved in \cite{LR95} (part of the proof of Prop.~6). It states that a function that is quasiconvex on $\S$ can be extended to a quasiconvex function on $\W_h$, via a combination of pullback and projection.

\begin{lemma}
\label{lm:qw0sigma qc}
Let $U: T^*\S\otimes\R^n\to\R$ be quasiconvex. Then,
$\pi^* U \circ \sigma^*: T^*\W_{h}\otimes\R^n \to \R$ is quasiconvex.
\end{lemma}

\begin{proof}
Fix $h$,  $\xi\in\W_h$ and $A\in T_\xi^*\W_h\otimes\R^n$ and let $D_\xi\subset T_\xi\W_h$ be some bounded set (which we will choose later). 
We {need to} prove that for every $\vp\in C_0^\infty(D_\xi;\R^n)$,
\[
\left(\pi^*U\circ \sigma^*\right)_\xi (A)\le \dashint_{D_\xi} \left(\pi^*U\circ \sigma^*\right)_\xi (A+d\vp\circ \kappa)\omega_\xi
\]
where $\kappa$ and $\omega_\xi$ are as in Definition \ref{df:QC}.

Denote 
\[
\begin{split}
T_\xi\W_h^\parallel & := \im \sigma_\xi = \sigma_\xi (T_{\pi(\xi)}\S),\\
T_\xi\W_h^\perp & := \im \iota_\xi = \iota_\xi (\NS_{\pi(\xi)}).
\end{split}
\]
Obviously, $T_\xi\W_h=T_\xi\W_h^\parallel \oplus T_\xi\W_h^\perp$, and denote by $\Pperp:T_\xi\W_h\to T_\xi\W_h^\perp$ and $\Ppar:T_\xi\W_h\to T_\xi\W_h^\parallel$ the projections.

Observe that $\omega_\xi=\opar \wedge\operp$, where $\opar$ and $\operp$ are respectively, $(n-k)$ and $k$ forms on $T_\xi\W_h^\parallel$ and $T_\xi\W_h^\perp$ (to simplify the notation, we identify  $\opar$ with $\Ppar^\star\opar$ and  $\operp$ with $\Pperp^\star\operp$).
Indeed, we can choose {local} coordinates on $\W_h$ {at $\xi$} such that the induced coordinates on $T_\xi\W_h$ are $(x_1,\ldots,x_{n-k},z_1,\ldots,z_k)$, where $(x_1,\ldots,x_{n-k})$ and $(z_1,\ldots,z_k)$ are bases for $T_\xi\W_h^\parallel$ and $T_\xi\W_h^\perp$, respectively.
In these  coordinates,
\[
\omega_\xi = \left. \Volume\right|_{\xi}\circ \Lambda^n \kappa^{-1}= 
a \,dx_1\wedge \ldots \wedge dx_{n-k} \wedge dz_1\wedge \ldots \wedge dz_k,
\]
where $a$ is some (constant) number. Now define $\opar := a\, dx_1\wedge \ldots \wedge dx_{n-k}$ and $\operp := dz_1\wedge \ldots \wedge dz_k$.

Choose $D_\xi$ such that $D_\xi = \dpar \times \dperp$, where $\dpar$ and $\dperp$ are bounded subsets of  $T_\xi\W_h^\parallel$ and $T_\xi\W_h^\perp$, respectively.
Restricting $\Pperp$ to $D_\xi$, we have
\beq
\label{eq:qclm}
\begin{split}
\int_{D_\xi} \left(\pi^*U\circ \sigma^*\right)_\xi (A+d\vp\circ \kappa)\omega_\xi &=
	\int_{\dperp} \Pperp_\star (\left(\pi^*U\circ \sigma^*\right)_\xi (A+d\vp\circ \kappa)\opar\wedge\operp)\\
	& = \int_{\dperp} \Pperp_\star (\left(\pi^*U\circ \sigma^*\right)_\xi (A+d\vp\circ \kappa)\opar)\wedge\operp\\
	& = \int_{\dperp} \Pperp_\star (\left(\pi^*U\right)_\xi (A\circ \sigma_\xi+d\vp\circ \kappa\circ \sigma_\xi)\opar)\wedge\operp.
\end{split}
\eeq

We now analyze the integral $\Pperp_\star (\left(\pi^*U\right)_\xi (A\circ \sigma_\xi+d\vp\circ \kappa\circ \sigma_\xi)\opar)$.
Let $(x,z)$ be a point in $D_\xi$, where $x\in \dpar$ and $z\in \dperp$.
Define $\vp_z:\dpar\to \R^n$ by $\vp_z(\cdot) := \vp(\cdot,z)$.
Let $\alpha\in \dpar$. We then have
\[
\begin{split}
d_{(x,z)}\vp\circ \kappa_{(x,z)}(\alpha) &= d_{(x,z)}\vp ([(x,z)+\alpha t]) = \frac{d}{dt}\vp\brk{(x,z)+\alpha t}\\
	&=\frac{d}{dt}\vp_z\brk{x+\alpha t} = d_x\vp_z([x+\alpha t] )= d_x \vp_z \circ \kappa_x (\alpha),
\end{split}
\]
which implies that
\beq
\label{eq:qclm1}
d_{(x,z)}\vp\circ \kappa_{(x,z)}\circ \sigma_\xi = d_{x}\vp_z\circ \kappa_{x}\circ \sigma_\xi,
\eeq
since the image of $\sigma_\xi$ is in $T_\xi\W_h^\parallel$. Since $\sigma_\xi$ is a linear mapping,
\[
d_y\sigma_\xi = \kappa_{\sigma_\xi(y)} \circ \sigma_\xi \circ (\kappa^\S_y)^{-1},
\]
where $y\in T_{\pi(\xi)}\S$ and  $\kappa_y^\S : T_{\pi(\xi)}\S \to T_yT_{\pi(\xi)}\S$ is the canonical identification.
Therefore, equation~\eqref{eq:qclm1} implies that
\[
d_{(x,z)}\vp\circ \kappa_{(x,z)}\circ \sigma_\xi = d_{\sigma_\xi^{-1}(x)}\brk{\vp_z\circ\sigma_\xi}\circ \kappa^\S_{\sigma_\xi^{-1}(x)}.
\]
Fixing $z$, we have that on the fiber $\dpar\times \{z\}$,
\[
\left(\pi^*U\right)_\xi (A\circ \sigma_\xi+d\vp\circ \kappa\circ \sigma_\xi)=
(\sigma^{-1}_\xi)^* \brk{U_{\pi(\xi)} \brk{A\circ \sigma_\xi+d\brk{\vp_z\circ\sigma_\xi}\circ \kappa^\S}}.
\]
Denoting by $p$ the mapping $\sigma_\xi^{-1}(\dpar)\to \{z\}$, we then have, using the change of variables formula $\Pperp_\star \brk{\sigma_\xi^{-1}}^\star = p_\star$,
\beq
\label{eq:qclm2}
\begin{split}
&\Pperp_\star (\left(\pi^*U\right)_\xi (A\circ \sigma_\xi+d\vp\circ \kappa\circ \sigma_\xi)\opar)\\
	&\quad = \Pperp_\star \brk{(\sigma^{-1}_\xi)^* \brk{U_{\pi(\xi)} \brk{A\circ \sigma_\xi+d\brk{\vp_z\circ\sigma_\xi}\circ \kappa^\S}}(\sigma^{-1}_\xi)^\star(\sigma^{-1}_\xi)_\star\opar}\\
	&\quad = p_\star \brk{U_{\pi(\xi)} \brk{A\circ \sigma_\xi+d\brk{\vp_z\circ\sigma_\xi}\circ \kappa^\S}(\sigma^{-1}_\xi)_\star\opar}.
\end{split}
\eeq
Since up to a constant, $(\sigma^{-1}_\xi)_\star\opar$ is just the form $\VolumeS|_{\pi(\xi)}\circ \Lambda^{n-k}(\kappa^\S)^{-1}$, the quasiconvexity of $U$ implies that
\beq
\label{eq:qclm3}
\begin{split}
&p_\star \brk{U_{\pi(\xi)} \brk{A\circ \sigma_\xi+d\brk{\vp_z\circ\sigma_\xi}\circ \kappa^\S}(\sigma^{-1}_\xi)_\star\opar}\\
	&\qquad \ge U_{\pi(\xi)} \brk{A\circ \sigma_\xi} p_\star \brk{(\sigma^{-1}_\xi)_\star\opar} =
		\left(\pi^*U\circ \sigma^*\right)_\xi(A) \Pperp_\star \brk{\opar},
\end{split}
\eeq
where in the last step we used again the change of variables formula. 

Combining equations~\eqref{eq:qclm2}-\eqref{eq:qclm3}, and inserting them into equation~\eqref{eq:qclm}, we have
\[
\begin{split}
\int_{D_\xi} \left(\pi^*U\circ \sigma^*\right)_\xi (A+d\vp\circ \kappa)\omega_\xi & \ge
	\left(\pi^*U\circ \sigma^*\right)_\xi(A) \int_{\dperp} \Pperp_\star \brk{\opar}\wedge\operp \\ 
	& = \left(\pi^*U\circ \sigma^*\right)_\xi(A) \int_{D_\xi} \omega_\xi,
\end{split}
\]
which completes the proof.
\end{proof}

\paragraph{Step 3: $I_0(F) \le I(F)$ when $F$ satisfies both the regularity and the boundary conditions}

\begin{proposition}
$I_0(F)\le I(F)$ for every $F\in W^{1,p}_{bc}(\S;\R^n)$.
\end{proposition}

\begin{proof}
Let $F\in W^{1,p}_{bc}(\S;\R^n)$ and consider a sequence $f_h\in W^{1,p}_{bc}(\W_h;\R^n)$ defined by
\[
f_h = F\circ\pi + \pi^* \qperp\circ\lambda,
\]
where $\qperp\in \{r\in \Gamma(\S; \NS^* \otimes\R^n): r|_{\partial\S}=\qperp_{bc}\}$ (a simple argument using a partition of unity of $\S$ shows that this set is non-empty). That is, for every $\xi\in\W_h$,
\[
f_h(\xi) = F(\pi(\xi)) + \qperp_{\pi(\xi)}(\xi).
\]

It is easy to see that $f_h\to F$ in $L^p$, hence by the lower-semicontinuity property of the $\Gamma$-limit:
\begin{equation}
\label{eq:I_0(F)1}
I_0(F) \le \liminf_{h\to 0} I_h(f_h) = \liminf_{h\to0} \dashint_{\W_h} W(df_h) \, \Volume.
\end{equation}
From Proposition 6.2 in \cite{KS14}, we have that
\[
|df_h\circ\Pi - \pi^*(dF\oplus\qperp)|\le Ch(1+|df_h|).
\]
It follows from the Lipschitz and the homogeneity over fibers properties of $W$ that 
\[
\begin{split}
\dashint_{\W_h} W(df_h) \, \Volume &= \dashint_{\W_h} \pi^*\left.W\right|_\S(df_h\circ\Pi) \, \Volume  \\
&= \dashint_{\W_h} \pi^*\left.W\right|_\S(\pi^*(dF\oplus\qperp)) \, \Volume + O(h) \\
	&=\int_{\S} \left.W\right|_\S(dF\oplus\qperp)\,\frac{\pi_\star\Volume}{|\W_h|} +O(h),
\end{split}
\]
where in the last step we used Fubini's theorem to first integrate over the fibers.
Using Lemma \ref{lm:cm} to evaluate $\pi_\star\Volume/|\W_h|$, Equation~\eqref{eq:I_0(F)1} reduces to
\[
I_0(F) \le \dashint_{\S} \left.W\right|_\S(dF\oplus\qperp)\, \VolumeS.
\]
This inequality holds for every $\qperp\in \{r\in \Gamma(\S; \NS^* \otimes\R^n): r|_{\partial\S}=\qperp_{bc}\}$.
Since $\{r\in \Gamma(\S; \NS^* \otimes\R^n): r|_{\partial\S}=\qperp_{bc}\}$ is dense in $L^p(\S; \NS^* \otimes\R^n)$, it follows from the Lipschitz property of $W$  that
\begin{equation}
\label{eq:I_0(F)2}
I_0(F) \le \inf_{\qperp\in L^p(\S; \NS^* \otimes\R^n)} \dashint_{\S} \left.W\right|_\S(dF\oplus\qperp)\, \VolumeS.
\end{equation}

By the definition of $W_0$, there exists a function $\qperp_0:\S \to \NS^* \otimes\R^n$ such that
\begin{equation}
\label{eq:msW0}
\left.W\right|_\S(dF\oplus \qperp_0) = W_0(dF).
\end{equation}
Thus, it seems that we can bound the infimum on the right hand side of  \eqref{eq:I_0(F)2} with $\dashint_{\S} W_0(dF)\, \VolumeS$. There is however one caveat: there is not a priori guarantee that $\qperp_0$ is measurable.  Lemma~\ref{lm:technical} below proves that there exists $\qperp_0\in L^p(\S;\NS^*\otimes\R^n)$ that satisfies  \eqref{eq:msW0}, hence for every $F\in  W^{1,p}_{bc}(\S; \R^n)$:
\begin{equation}
\label{eq:I_0(F)3}
I_0(F) \le \dashint_{\S} W_0(dF)\, \VolumeS \equiv G(F).
\end{equation}

We introduce the following notation: for a function $H: W^{1,p}_{bc}(\S; \R^n)\to \R$ we set  $\wt{H}: L^p(\S;\R^n) \to \R$ to be
\[
\wt{H}(F) =
\begin{cases}
H(F) & F\in W^{1,p}_{bc}(\S; \R^n) \\
\infty & \text{otherwise}.
\end{cases}
\]
Equation \eqref{eq:I_0(F)3} implies that $I_0(F)\le \wt{G}(F)$ for every $F\in L^p(\S;\R^n)$.
Since, moreover, $I_0$ is a $\G$-limit, it is lower-semicontinuous with respect to the strong $L^p$-topology, and therefore $I_0\le \G \wt{G}$, where $\Gamma \wt{G}$ denoted the lower-semicontinuous envelope (with respect to the same topology) of $\wt{G}$.

Next denote by $\G_w G$ the sequential lower-semicontinuous envelope of $G$ with respect to the \emph{weak} topology in $W^{1,p}_{bc}(\S; \R^n)$. Lemma 5 in \cite{LR95} implies that $\wt{\Gamma_w G} = \Gamma \wt{G}$, hence $I_0\le \wt{\G_w G}$; in particular, for $F\in W^{1,p}_{bc}(\S;\R^n)$, $I_0(F) \le \G_w G(F)$.
Finally, it follows from \cite{AF84} that 
\[
\G_w G(F) = \dashint_{\S} QW_0(dF)\, \VolumeS = I(F),
\]
which completes the proof.
\end{proof}

\begin{lemma}
\label{lm:technical}
There exists $\qperp_0\in L^p(\S;\NS^*\otimes\R^n)$ that satisfies
\[
\left.W\right|_\S(dF\oplus \qperp_0) = W_0(dF).
\]
\end{lemma}

\begin{proof}
We first prove that there exists a measurable $\qperp_0:\S\to\NS^*\otimes\R^n$ such the equality above holds, and then prove that it is in $L^p$.

Define $\NW = W|_\S(dF\oplus\cdot): \NS^* \otimes\R^n \to \R$.
We are looking for a measurable section $\qperp_0$ that minimizes $\NW$ on every fiber.
The measurable selection theorem (Theorem~\ref{tm:mst}) deals with the existence of such measurable sections.
However, while $\NW$ satisfies the regularity assumptions in the theorem
(it is a Carath\'eodory function as (i) it is fiber-wise continuous due to the continuity of $W|_\S$, and (ii) for every smooth section $\eta\in\G(\S;\NS^* \otimes\R^n)$,  $\NW\circ\eta$ is measurable as a composition of a continuous and a measurable function (see Appendix~\ref{ap:mst})),
the measurable selection theorem  cannot be applied directly to $\NW$, as the fiber is a vector space and hence not compact.
We therefore proceed by using the coercivity and growth properties of $W$ to overcome this problem of non-compactness.

The coercivity and growth properties imply that for every $x\in\S$ and $\xi\in \NS^* \otimes\R^n$,
\[
\alpha |\xi|^p - \beta < \NW(\xi)
\]
and
\[
\inf_{\eta\in \brk{\NS^* \otimes\R^n}_x}  \NW(\eta) < C(1 + |d_x F|^p).
\]
Hence, if we denote by $\S_N$ the set $\{x\in \S : |d_x F|^p < N\}$, the measurable selection theorem (Corollary~\ref{cy:mst}) can be applied to $\NW|_{\S_N}$, since the minimizer on every fiber lies in a ball of radius $\brk{\frac{C(1+N)+\beta}{\alpha}}^{1/p}$. Denote this measurable minimizer by $\qperp_{0,N}$, and construct $\qperp_0$ by
\[
\qperp_0(x) = \qperp_{0,N}(x), \qquad N=\min\{M\in\mathbb{N} : x\in \S_M\}.
\]
Since, up to a null set, $\cup_N \S_N = \S$, $\qperp_0$ is well defined on almost every point in $\S$. It is obviously measurable, since for every $N$, $\qperp_{0,N}$ is measurable.

Finally, $\qperp_0\in L^p(\S; \NS^* \otimes\R^n)$ since
\[
\begin{split}
\alpha\int_\S |\qperp_0|^p\,\VolumeS -\beta\cdot \Vol(\S) &\le \int_\S W|_\S (dF\oplus \qperp_0)\,\VolumeS \\
	&\le \int_\S W(dF\oplus 0 )\, \VolumeS \\
	&\le C\int_\S (1 + |dF|^p) \,\VolumeS < \infty.
\end{split}
\]
\end{proof}

We have thus completed the proof of Theorem~\ref{tm:main}. We finish this section by a proof of {the} main corollary:

\begin{quote}
\emph{
Let $f_h\in W_{bc}^{1,p}(\W_h;\R^n)$ be a sequence of (approximate) minimizers of $I_h$. Then $(f_h)$ is a relatively compact sequence (with respect to the strong $L^p$ topology), and all its limits points are minimizers of $I$.
Moreover,
\[
\lim_{h\to0} \,\,\,\inf_{L^p(\W_h;\R^n)} I_h(\cdot) = \min_{L^p(\S;\R^n)} I(\cdot).
\]
}
\end{quote}

\begin{proof}
Let $f_h$ be a sequence of approximate minimizers of $I_h$. 
We first prove that it is relatively compact, i.e., that every subsequence (not relabeled) of $f_h$ has a subsequence that converges in $L^p$.

Let $g\in W_{bc}^{1,p}(\S;\R^n)$ be arbitrary and let $g_h\in L^p(\W_h;\R^n)$ be a recovery sequence for $g$. Then,
\[
\inf_{L^p} \,\,I_h(\cdot) \le I_h(g_h) \limarrow_{h\to0} I(g)  <\infty ,
\]
due to the growth property of $QW_0$.  
This shows  that $\inf_{L^p} \,\,I_h(\cdot)$ is bounded.

It follows that $I_h(f_h)$ is bounded, hence by coercivity $df_h$ is uniformly bounded in $L^p$, and together with the Poincar\'e inequality,   $f_h$ is uniformly bounded in $W^{1,p}$. Lemma \ref{lm:wk w1p conv}  implies the existence of a  subsequence $f_{h}\to F$ in $L^p$, proving the relative compactness of $f_h$.

We now prove that $F$ is a minimizer of $I$. Let $g\in L^p(\S;\R^n)$ be an arbitrary function, and  let $g_h\in L^p(\W_h;\R^n)$ be a recovery sequence for $g$. Therefore,
\[
I(g) = \lim_{h\to0} I_h(g_h) \ge \lim_{h\to0} \inf_{L^p} I_h(\cdot) = \lim_{h\to0} I_h(f_h) \ge I(F),
\]
where the last inequality follows from the lower-semicontinuity property of $\G$-convergence.
Since $g$ is arbitrary, $F$ is a minimizer of $I$.
Moreover, by choosing $g=F$ we conclude that
\[
I(F)  = \lim_{h\to0} \inf_{L^p} I_h(\cdot).
\]
\end{proof}

\section{Discussion}

This paper generalizes the work of Le Dret and Raoult \cite{LR95,LR96} to a general Riemannian setting and general dimension and co-dimension, hence is applicable to slender pre-stressed bodies. We now emphasize the main differences between the present analysis and the prior work that was derived in the Euclidean setting; these can be partitioned into analytical issues and modeling issues. 

\paragraph{Analytical issues}

Thin bodies are modeled as a family of tubular neighborhoods $\W_h$ of a Riemmanian manifold $(\calM,\go)$, that converges to a lower-dimensional submanifold $\S$. Accordingly, we defined a notion of convergences of functions $L^p(\W_h;\R^n)$ to functions $L^p(\S;\R^n)$.  The fact that configurations for different $h$ are defined over different functional spaces requires only minor adaptations in the $\G$-convergence approach, since the latter is not affected by the Riemannian structure. Other analytic notions, however, such as quasiconvexity and measure theoretic issues require a more detailed attention. 

The $\G$-limit of a sequence of functionals is lower-semicontinuous. As weak lower-semicontinuity of an integral functional is closely related to the quasiconvexity of the integrand, we had to properly define the notion of quasiconvexity of functions over manifolds (Definition~\ref{df:QC}) and show that the classical results in \cite{AF84} remains valid in this settings (Appendix~\ref{appendix:qc}). As in the Euclidean case, the limit energy density $QW_0$ is the quasiconvex envelope of $W_0$, which is a projection of the original energy density $W$ to the limiting submanifold.

For the energy functional to be well-defined, the density has to be sufficiently regular. We assume $W$ to be continuous and show that this also implies the continuity of $W_0$ (Lemma~\ref{lm:w0}). The more challenging step is to show that $QW_0$ is sufficiently regular, and specifically, a Carath\'eodory function. To do so, we need to properly define  Carath\'eodory functions over fiber bundles  (see Appendix~\ref{ap:mst}), and show that the quasiconvex envelope of a Carath\'eodory function over fiber bundles is again a Carath\'eodory function (Theorem~\ref{tm:qc2} and Corollary~\ref{cy:qw0}).
This notion of Carath\'eodory functions is also needed for the theorems regarding quasiconvexity and lower-semicontinuity mentioned above. 

Finally, a generalization of the related notion of normal integrand to functions over fiber bundles enables us to prove a generalization of a measurable selection theorem, which shows the existence of a measurable minimizing section (Theorem~\ref{tm:mst} and Corollary~\ref{cy:mst}), which is needed in the proof of Lemma~\ref{lm:technical}.

\paragraph{Modeling issues and properties of the limit energy density}

The main assumption on the energy density $W$ is its homogeneity over fibers (other than that, there are only regularity assumptions). 
In particular, we do not assume frame-indifference or isotropy. However, if the energy density does satisfy either of them, so does the limit energy density $QW_0$; the proofs are essentially the same as in Theorems 9 and 13 in \cite{LR95}. 
Since our setting is not necessarily Euclidean, isotropy here means invariance of the energy density under orientation preserving isometries of the relevant manifold ($(\calM,\go)$ in the case of $W$, $(\S,\go|_\S)$ in the case of $QW_0$).
If $W$ satisfies frame indifference, the frame indifference of $QW_0$ implies that the limit functional $I$ depends on an immersion $F\in W^{1,p}_{bc}(\S;\R^n)$, only through the pullback metric on $\S$ induced by $F$, that is $F^\star \euc$. In other words, as expected, the only contribution to the membrane energy is from stretching of the limiting submanifold, in contrast to bending dominated limits, such as in \cite{KS14}.

The absence of bending contributions to the energy holds even without assuming frame indifference, as the limiting energy functional depends only on first derivatives of an immersion of the limiting submanifold $\S$. 
Moreover, the membrane energy does not ``know'' whether it is a limit of a Euclidean or a non-Euclidean problem, in the following sense. Since $\dim \S < n$, the Nash-Kuiper embedding theorem implies that $\S$ can be $C^1$-isometrically embedded in the Euclidean space $\R^n$. 
Consider $(\S,\go|_\S)$ as a sub-manifold of $\R^n$, and
let $\W_h'$ be its tubular neighborhoods; they are Euclidean shells. We may then define an energy density $W'$ on 
$\W_h'$, with $(W')_0=W_0$. The limiting membrane model for $\S$ as a submanifold of $(\R^n,\euc)$ with energy density $W'$ is the same as the one obtained for $\S$ as a submanifold of $(\calM,\go)$ with energy density $W$.

The homogeneity over fibers assumption enables us to relate between energy densities on the limit submanifold $\S$ and on the tubular neighborhoods $\W_h$, and is therefore essential to the proof of Proposition \ref{pn:I_0>I} (it can be slightly weakened as long as Corollary \ref{cy:w sigma iota} holds for some positive power of $h$).
Indeed, in \cite{BFF00} and \cite{BF05} it is shown that for Euclidean plate membranes the limit energy density is substantially different (and more complicated) when the inhomogeneity is in the normal directions (when choosing coordinates, homogeneity over fibers allows inhomogeneity only in the tangent directions).
Note also that while homogeneity or inhomogeneity is usually stated in a specific coordinate system, our definition does not depend on the coordinate system, hence it reveals the geometric essence of this notion. 

The stronger and more common notion of homogeneity (which was assumed in \cite{LR95,LR96}), implies that the energy density at one point determines it everywhere. This notion can also be generalized to the non-Euclidean case via parallel transport.
Note, however, that unlike the Euclidean case,  parallel transport is generally path-dependent. Thus, while in the Euclidean case every energy density at a point extends to a homogeneous energy density over the entire manifold, this is not so in the non-Euclidean case (however, there are homogeneous energy densities for every metric; the prototypical energy density \eqref{eq:Wproto} is such).
This issue does not arise in our analysis, as homogeneity over fibers implies a specific choice of non-intersecting paths (the fibers) along which parallel transport is calculated.

The results of this paper are also applicable to slender bodies of varying thickness, considered in \cite{BFF00}.
Varying thickness in this case means that the thickness of $\W_h$ is equal to $h$ times some continuous function $t:\S\to(0,\infty)$. In this case $\VolumeS$ in the limiting density has to be replaced by $t\,\VolumeS$.

\appendix

\section{Measurability issues on Riemannian manifolds}
\label{ap:mst}

The main result in this section is a generalization of a measurable selection theorem to fiber bundles over manifolds. First, we give some basic definitions.

\subsection{Normal integrands and Carath\'eodory functions}

\begin{definition}
Let $\calM$ be an $m$-dimensional differentiable manifold. A function $f:\calM\to\bar{\R}$ is \emph{measurable} if for every chart $X:U\subset\R^m\to\calM$, $f\circ X:U\to\bar{\R}$ is (Lebesgue) measurable.
Similarly, if $\calN$ is an $n$-dimensional differentiable manifold, a function $f:\calM\to\calN$ is measurable if for every chart $Y:V\subset\R^n\to\calN$, $Y^{-1}\circ f\circ X$ is (Lebesgue) measurable on its domain of definition.
\end{definition} 

A standard argument shows that it is enough to check measurability over an atlas. This definition gives a natural notion of Lebesgue-induced measurable subsets of a differentiable manifold (which obviously contain the Borel $\sigma$-algebra); for short, we will call these subsets Lebesgue measurable. We may therefore extend our definition of measurability to functions $f:A\to\N$, where $A$ is a (measurable) subset of $\calM$. Again, similar to the Euclidean case, if $\{A_i\}$ is a countable measurable partition of $\calM$, and $f_i:A_i\to\N$ are measurable for every $i$, then $f=\cup f_i : \calM \to \N$ is a measurable function.

A Riemannian metric on the differentiable manifold $\calM$ induces a measure on this Lebesgue $\sigma$-algebra. As in the Euclidean case, the corresponding $L^p$ space coincides with the completion of the smooth functions with {respect} to the $L^p$ norm.

We now define the notions of normal integrands and Carath\'eodory functions.
Let $F\subset \R^l$ be a Borel set, and let $\pi:E\to\calM$ be a fiber bundle with fiber $F$. we define a ``hybrid" $\sigma$-algebra on $E$ as the $\sigma$-algebra generated by the Borel $\sigma$-algebra on $F$ and the Lebesgue $\sigma$-algebra on $\calM$.

\begin{definition}
Let $\pi:E\to\calM$ be as above. A function $f:E\to\bar{\R}$ is a \emph{normal integrand} if
\begin{enumerate}
\item $f$ is measurable with respect to the hybrid $\sigma$-algebra.
\item for almost every $p\in\calM$, $f|_{E_p}$ is lower-semicontinuous.
\end{enumerate}
\end{definition}

A normal integrand satisfies the following property: if $\rho$ is a measurable section of $E$, then $f\circ\rho:\calM\to\bar{\R}$ is measurable.
Note that the definition of a normal integrand can be extended to the case where $\pi:E\to A$ is a vector bundle over some measurable subset $A\subset\calM$.

\begin{definition}
Let $\pi:E\to\calM$ be as above. A function $f:E\to\bar{\R}$ is a \emph{Carath\'eodory function} if
\begin{enumerate}
\item for every smooth section $\rho$ of $E$, $f\circ\rho:\calM\to\bar{\R}$ is measurable.
\item for almost every $p\in\calM$, $f|_{E_p}$ is continuous.
\end{enumerate}
\end{definition}

These definitions coincide with the classical definitions of normal integrands and Carath\'eodory functions over the trivial bundles $\R^n\times F$ (see e.g. \cite{ET76} or \cite{RW09}). A standard argument, based on a local trivialization of the bundle,  shows that a function is a normal integrand (resp. Carath\'eodory function) if and only if it is locally a normal integrand (resp. Carath\'eodory function) in the classical sense.
It can therefore be deduced that a function $f$ is Carath\'eodory if and only if both $f$ and $(-f)$ are normal integrands.

\subsection{A measurable selection theorem}

\begin{theorem}
\label{tm:mst}
Let $F\subset R^l$ be a compact set, and let $\pi:E\to\calM$ be a fiber-bundle with fiber $F$.
Let $f: E\to\bar{\R}$ be a normal integrand.
Then there exists a measurable section $\rho:\calM\to E$ such that for almost every $p\in \calM$
\[
f\circ \rho (p) = \min_{\xi\in E_p} \{ f(\xi)\}.
\]
\end{theorem}

\begin{proof}
First, we cover $\calM$ with a countable atlas $X_i : U_i \to \calM$, where $U_i$ are open sets in $\R^n$, and let $Y_i : U_i \times F \to E$ be a corresponding atlas of $E$.
By the measurable selection theorem for trivial bundles (see e.g. \cite{ET76}), there exists a measurable $v_i : U_i\to F$ such that
\[
f\circ Y_i (x, v_i (x)) = \min_{a\in F} f\circ Y_i (x,a) 
\qquad \forall  x\in U_i
\]
Define $\rho_i(p) = Y_i(X_i^{-1}(p), \nu_i \circ X_i^{-1} (p))$. This is a measurable section of $E|_{X_i(U_i)}$ that satisfies
\[
f\circ \rho_i (p) = \min_{\xi\in E_p} \{ f(\xi)\} 
\qquad \forall p\in X_i(U_i).
\]
Define next a section $\rho:\calM\to E$ by
\[
\rho(p)=\rho_N(p), \qquad N=\min\{i\in\mathbb{N} : p\in U_i\}.
\]
$\rho$ is obviously measurable and satisfies $f\circ \rho (p) = \min_{\xi\in E_p} \{ f(\xi)\}$ for almost every $p\in\calM$.
\end{proof}

\begin{corollary}
\label{cy:mst}
The theorem also holds for a normal integrand $f':E'\to \bar{\R}$, where $\pi':E'\to A$ is a fiber bundle with a compact fiber $F$ over a base $A$ which is a measurable subset of $\calM$.
\end{corollary}

\begin{proof}
Extend $f'$ to $f:E\to\bar{\R}$ by defining it to be $-\infty$ on $E\setminus E'$. $f$ is a normal integrand, hence we can apply the theorem {to} $f$, and then restrict the resulting $\rho:\calM\to E$ to $A$.
\end{proof}

\section{On quasiconvexity and lower-semicontinuity}
\label{appendix:qc}

In this section we prove Theorems~\ref{tm:qc1}--\ref{tm:qc2} that generalize classical results on the relation between quasiconvexity and lower-semicontinuity in the Riemannian setting.
The proofs are basically applications of the main results of \cite{AF84}. 
Let $(\calM,\go)$ be a Riemannian manifold of finite volume that can be covered by a finite number of charts, let  $U: T^*\calM\otimes\R^m \to \R$ be a Carath\'eodory function that satisfies the growth condition $-\beta \le U(q) \le C\brk{1+|q|^p}$. 
The growth condition and the finite number of charts assumption imply that when choosing coordinates, $U$ satisfies the growth assumptions (II.2) and (II.6)  in \cite{AF84} (actually, $U+\beta$ satisfies the conditions, but since we assumed that $\calM$ {has} finite volume, the addition of a constant to the integrand is immaterial).
First we prove Theorem~\ref{tm:qc1}:
\begin{quote}
\emph{
Under the above conditions , the functional $I_A:W^{1,p}(A;\R^m)\to \R$ defined by
\[
I_A(f) := \int_A U(df)\,\,\Volume,
\]
where $A\subset\calM$ is an open subset, is weakly sequential lower-semicontinuous for every $A$ if and only if $U$ is quasiconvex.
}
\end{quote}

\begin{proof}
First assume that $I_A$ is weakly sequential lower-semicontinuous in $W^{1,p}(A;\R^m)$ for some open $A\subset\calM$ contained in a coordinate neighborhood. 
Using coordinates, we apply Theorem [II.2] in \cite{AF84}, and get that $U$ is quasiconvex in $ T^*A\otimes\R^m$. Since quasiconvexity is a fiber-wise condition, if $U|_A$ is quasiconvex for every open $A\subset\calM$, then $U$ is quasiconvex (as in Definition~\ref{df:QC}).

Now assume that $U$ is quasiconvex. Let $f,f_n\in W^{1,p}(\calM;\R^m)$ such that $f_n\weakly{}f$. We can partition $\calM$, up to a null-set, into a finite number of disjoint open sets $\{A_i\}$ such that for every $i$, $A_i$ is contained in a coordinate neighborhood. 
Since $U$ is a Carath\'eodory quasiconvex function and satisfies the growth condition \eqref{eq:growthqc}, it follows from Theorem [II.4] in \cite{AF84} that $I_{A_i}$ is weakly sequential lower-semicontinuous in $W^{1,p}(\calM;\R^m)$. Therefore,
\[
\liminf_{n\to\infty} I_\calM(f_n) = \sum_i \liminf_{n\to\infty} I_{A_i}(f_n|_{A_i}) \ge \sum_i I_{A_i}(f|_{A_i}) = I_\calM(f),
\]
which shows that $I_{\calM}$ is weakly sequential lower-semicontinuous, and therefore $I_A$ is weakly sequential lower-semicontinuous for every open $A\subset\calM$.
\end{proof}

Next we prove Theorem~\ref{tm:qc2}:
\begin{quote}
\emph{Under the above conditions, the weakly sequential lower-semicontinuous envelope of the functional $I_\calM:W^{1,p}(\calM;\R^m)\to \R$ is $\G I_\calM:W^{1,p}(\calM;\R^m)\to \R$ given by
\[
\G I_\calM(f) := \int_\calM QU(df)\,\,\Volume,
\]
where $QU(q)=\sup\{V(q): V\le U \,\,\text{is quasiconvex}\}$ is the quasiconvex envelope of $U$; moreover $QU$ is a Carath\'eodory quasiconvex function.
}
\end{quote}

\begin{proof}
Let $A\subset\calM$ be a coordinate neighborhood. We apply Statement [III.7] in \cite{AF84} to get that the weakly sequential lower-semicontinuous envelope of $I_A$ is 
\[
\G I_A(f) := \int_A Q(U|_A)(df)\,\,\Volume.
\]
Because quasiconvexity is a fiberwise condition, the restriction of the quasiconvex envelope is the quasiconvex envelope of the restriction, and therefore we can replace $Q(U|_A)$ with $(QU)|_A$. Theorem [III.6] in \cite{AF84} implies that $(QU)|_A$ is a Carath\'eodory function, and since being a Carath\'eodory function is a local condition (see Appendix~\ref{ap:mst}), we have that $QU$ is indeed Carath\'eodory.
It is also quasiconvex, since generally the supremum of quasiconvex functions is quasiconvex (again, this is a fiberwise argument, and therefore it is true since it holds in the Euclidean case). 
By Theorem~\ref{tm:qc1}, $\G I_\calM$ is weakly sequential lower-semicontinuous, hence it is {bounded by} the sequentially lower-semicontinuous envelope of $I_\calM$, denoted by $\tI_\calM$.

Next, for a given $f\in W^{1,\infty}(A;\R^m) \subset W^{1,p}(A;\R^m)$, Statement [III.7] also implies that for every $\e>0$ there exists a sequence $f_n\in f + W^{1,\infty}_0(A;\R^m)$ such that $f_n-f\weakly{*}0$ in $W^{1,\infty}_0(A;\R^m)$ and $\liminf_n I_A(f_n) \le \G I_A(f) + \e$. This follows from the fact that in the notation of \cite{AF84},
\[
\G I_A(f) = \lim_{r\to\infty} F_0(r,f,A) = \lim_{r\to\infty} \Brk{\inf \{\liminf_{n\to\infty} I_A(f_k): f_n-f\weakly{*}0 \,\,\text{in}\,\, W^{1,\infty}_0(A;\R^m) \,\,\text{and}\,\, |du_n|_\infty\le r\}}.
\]
See also Theorem 3.8 in \cite{MS80}.

For the reverse inequality, let $\{A_i\}$ be a finite  partition (up to a null-set) of $\calM$, such that for every $i$, $A_i$ is an open set contained in a coordinate neighborhood, and let $f\in W^{1,\infty}(\calM;\R^m)$. Fix $\e>0$.
For every $i$, let $f^n_i \in f|_{A_i} + W^{1,\infty}_0(A_i;\R^m)$ be a sequence such that $f^n_i - f|_{A_i} \weakly{*}_{n}0$ in $W^{1,\infty}_0(A_i;\R^m)$ and
\[
\liminf_{n\to\infty} I_{A_i}(f^n_i) \le \G I_A(f|_{A_i}) + \frac{\e}{2^i}.
\]
Obviously, $f^n:=\cup_i f^n_i \weakly{*} f$ in $W^{1,\infty}(\calM;\R^m)$, and therefore
\[
\tI_\calM(f) \le \liminf_{n\to\infty} I_\calM(f^n) = \liminf_{n\to\infty} \sum_i I_{A_i}(f^n_i) \le  \sum_i \G I_A(f|_{A_i}) + \e = \G I_\calM (f) + \e,
\]
and since $\e$ is arbitrary, it shows that $\tI_\calM(f) \le \G I_\calM (f)$ for every $f\in W^{1,\infty}(\calM;\R^m)$. To extend this to $W^{1,p}(\calM; \R^m)$, observe that $W^{1,\infty}(\calM;\R^m)$ is dense in $W^{1,p}(\calM;\R^m)$ with respect to the strong $W^{1,p}$ topology, and that $\G I_\calM$ is continuous in $W^{1,p}(\calM;\R^m)$ with respect to this topology (see Proposition~\ref{pn:strcont} below).
Hence, given $f\in W^{1,p}(\calM;\R^m)$, let $f_n\in W^{1,\infty}(\calM;\R^m)$ such that $f_n\to f$, and we obtain that
\[
\tI_\calM(f) \le \liminf_{n\to\infty} \tI_\calM(f_n) =  \liminf_{n\to\infty} \G I_\calM(f_n) = \G I_\calM(f),
\]
which completes the proof.

\end{proof}

The strong continuity of $\G I_\calM$ follows from the following generalization of the Carath\'eodory continuity theorem:
\begin{proposition}
\label{pn:strcont}
Let $\calM$ and $U$ satisfy the assumptions as above. Then the functional $I_\calM:W^{1,p}(\calM;\R^m)\to\R$ is continuous with respect to the strong $W^{1,p}$ topology.
\end{proposition}

\begin{proof}
This is an immediate consequence of the analogous Euclidean proposition (see e.g. \cite{Dal93}, Example 1.22).
Indeed, let $A_i$ be an open partition (up to a null set) of $\calM$, such that for every $i$, $A_i$ is contained in a coordinate neighborhood. 
For every $i$, $U|_{A_i}$ is a Carath\'eodory function that satisfies \eqref{eq:growthqc}, since $U$ does. Hence, by the Carath\'eodory continuity theorem, $I_{A_i}$ is strongly continuous in $W^{1,p}(A_i;\R^m)$.
Let $f,f_n\in W^{1,p}(\calM;\R^m)$ such that $f_n\to f$. In particular, for every $i$, $f_n|_{A_i} \to f|_{A_i}$ in $W^{1,p}(A_i;\R^m)$.
Therefore,
\[
\lim_{n\to\infty} I_\calM(f_n) = \lim_{n\to\infty} \sum_i I_{A_i}(f_n|_{A_i}) = \sum_i I_{A_i}(f|_{A_i}) = I_\calM (f).
\]
\end{proof}


\begin{acknowledgements}
We are grateful to R.V. Kohn for many discussions and useful advice.
\end{acknowledgements}

\bibliographystyle{spmpsci}      
\bibliography{arXiv_version.bbl}

\end{document}